\documentclass[11pt,reqno]{amsart}

\usepackage{geometry}
\usepackage[hidelinks]{hyperref}

\usepackage{amsmath,amsthm,amssymb,amsfonts,bbm}
\usepackage{xcolor}

\theoremstyle{plain}
\newtheorem{theorem}{Theorem}[section]
\newtheorem{proposition}{Proposition}[section]
\newtheorem{lemma}{Lemma}[section]
\newtheorem{corollary}{Corollary}[section]
\newtheorem{assumption}{Assumption}[section]

\theoremstyle{definition}

\newtheorem{remark}{Remark}[section]

\newcommand\independent{\protect\mathpalette{\protect\independenT}{\perp}}
\def\independenT#1#2{\mathrel{\rlap{$#1#2$}\mkern2mu{#1#2}}}

\newcommand{\R}{\mathbb{R}}

\newcommand{\E}{\mathbb{E}}

\newcommand{\calB}{\mathcal{B}}
\newcommand{\calL}{\mathcal{L}}

\newcommand{\calX}{\mathcal{X}}
\newcommand{\calY}{\mathcal{Y}}
\newcommand{\calC}{\mathcal{C}}
\newcommand{\calP}{\mathcal{P}}

\newcommand{\calS}{\mathcal{S}}
\newcommand{\calM}{\mathcal{M}}
\newcommand{\calQ}{\mathcal{Q}}
\newcommand{\calU}{\mathcal{U}}
\newcommand{\calZ}{\mathcal{Z}}
\newcommand{\calN}{\mathcal{N}}

\newcommand{\bN}{\mathbb{N}}

\DeclareMathOperator{\tr}{tr}

\DeclareMathOperator{\Ent}{Ent}

\newcommand{\kk}[1] {{\color{red} KK: [#1]}}

\newcommand{\KL}[2]{\mathsf{KL}(#1 \,\|\, #2 )}

\newcommand{\frakf}{\mathfrak{f}}
\newcommand{\frakg}{\mathfrak{g}}
\newcommand{\frakh}{\mathfrak{h}}

\begin{document}
\title[Entropic VQR]{Entropic vector quantile regression: Duality and Gaussian case}

\thanks{K. Kato is partially supported by NSF grant DMS-2413405.}

\author[K. Kato]{Kengo Kato}
\address[K. Kato]{
Department of Statistics and Data Science, Cornell University.
}
\email{kk976@cornell.edu}

\author[B. Wang]{Boyu Wang}
\address[B. Wang]{
Department of Statistics and Data Science, Cornell University.
}
\email{bw563@cornell.edu}

\begin{abstract}
Vector quantile regression (VQR) is an optimal transport (OT) problem subject to a mean-independence constraint that extends classical linear quantile regression to vector response variables.  Motivated by computational considerations, prior work has considered entropic relaxation of VQR, but its fundamental structural and approximation properties are still much less understood than entropic OT. The goal of this paper is to address some of these gaps. First, we study duality theory for entropic VQR and establish strong duality and dual attainment for marginals with possibly unbounded supports. In addition, when all marginals are compactly supported, we show that dual potentials are real analytic. Second, building on our duality theory, when all marginals are Gaussian, we show that entropic VQR has a closed-form optimal solution, which is again Gaussian, and establish the precise approximation rate toward unregularized VQR.
\end{abstract}
\date{First version: February 9, 2026. This version: \today}

\maketitle

\section{Introduction}
\subsection{Overview}
Quantile regression \cite{koenker1978regression,koenker2005quantile} offers a powerful statistical methodology for modeling conditional quantiles for scalar response variables.
Recent attention in the literature is to extend 
 quantile regression to vector response variables, which is, however, not straightforward because of the lack of natural ordering in multidimensional space \cite{hallin2022measure}. Optimal transport (OT) provides a promising approach to extending the quantile function to vector variables.\footnote{We refer the reader to \cite{villani2009optimal,santambrogio2015optimal} as standard references on OT theory.} For a vector response variable $Y \in \R^{d_y}$ and a given absolutely continuous reference distribution $\mu$ on $\R^{d_y}$ (such as $\mu = \mathrm{Unif} ([0,1]^{d_y})$ or $\mu=\calN(0,I_{d_y})$), several authors proposed using the \textit{Brenier map} \cite{brenier1991polar} transporting $\mu$ onto the law of $Y$  as a vector quantile function for $Y$ \cite{carlier2016vector,chernozhukov2017monge,ghosal2022multivariate}.

In the presence of covariates $X \in \R^{d_x}$,  \cite{carlier2016vector} considered the conditional Brenier map $\Phi: \R^{d_y+d_x} \to \R^{d_y}$, transporting $\mu$ onto the conditional distribution of $Y$ given $X$, as a conditional vector quantile function, which satisfies the distributional relation
\[
(X,Y) \stackrel{d}{=} (X,\Phi(U,X)), \ U \mid X \sim \mu,
\]
where $\stackrel{d}{=}$ denotes equality in distribution. 
The joint distribution $\pi^o$ of $(U,X,\tilde{Y})$ with $\tilde{Y}=\Phi(U,X)$ and $U \mid X \sim \mu$ can be characterized as an optimal solution to the OT problem with an independence constraint, 
\begin{equation}
\inf_{\pi \in \Pi(\mu,\nu)}\big \{ \E[\|U-\tilde{Y}\|^2] :(U,\tilde{X},\tilde{Y}) \sim \pi, \ U \independent \tilde{X} \big \},
\label{eq: vqr indep}
\end{equation}
where $\nu$ denotes the distribution of $(X,Y)$ and $\Pi(\mu,\nu)$ denotes the collection of couplings for $(\mu,\nu)$. Recall that any coupling $\pi \in \Pi(\mu,\nu)$ is a joint distribution on $\R^{d_y} \times \R^{d_x+d_y}$ with marginals $\mu,\nu$. In addition, \cite{carlier2016vector} considered modeling the conditional vector quantile function $\Phi(u,x)$ as an affine function in $x$, i.e., 
\begin{equation}
\Phi(u,x) = b_0(u) + b_1(u)^\top x
\label{eq: affine}
\end{equation}
for suitable mappings $b_0: \R^{d_y} \to \R^{d_y}$ and $b_1:\R^{d_y} \to \R^{d_x \times d_y}$. 
When $\Phi$ is of the form (\ref{eq: affine}), under technical conditions, the corresponding coupling $\pi^o$ is optimal for the relaxed problem
\begin{equation}
\inf_{\pi \in \Pi(\mu,\nu)} \big \{\E[\|U-\tilde{Y}\|^2] : (U,\tilde{X},\tilde{Y}) \sim \pi, \, \E[\tilde{X} \mid U]=\E[X] \big \},
\label{eq: vqr init 0}
\end{equation}
which is an OT problem subject to a mean-independence constraint \cite[Theorem 3.1]{carlier2016vector}. Following \cite{carlier2016vector}, we shall call (\ref{eq: vqr init 0}) the \textit{vector quantile regression} (VQR) problem, which extends (linear) quantile regression to vector response variables (\cite[Theorem 3.3]{carlier2016vector}; see also their follow-up work \cite{carlier2017vector}). 

For standard OT, entropic regularization offers various computational and analytical advantages, which has prompted extensive research activities on entropic OT \cite{leonard2013survey,nutz2021introduction}. The popularity of entropic OT stems from the fact that it is amenable to efficient computation via a dual block coordinate ascent algorithm, called \textit{Sinkhorn’s algorithm}, for which rigorous convergence guarantees have been developed \cite{sinkhorn1967diagonal,franklin1989scaling,cuturi2013sinkhorn,peyre2019computational,ghosal2025convergence}. Dual potentials for entropic OT, which are completely characterized by the system of functional equations, called the \textit{Schr\"{o}dinger system}, are smooth provided that the ground cost is smooth and marginals have  sufficiently light tails, which enables faster sample complexity rates than unregularized OT  \cite{genevay2019sample, mena2019statistical,del2023improved, gonzalez2022weak,goldfeld2024limit}. In addition, as the regularization parameter tends to zero, various objects of entropic OT converge to those of unregularized OT \cite{mikami2004monge,leonard2012schrodinger,nutz2022entropic,carlier2023convergence,eckstein2024convergence,malamut2025convergence}

For VQR, one can consider its entropic variant by adding an entropic penalty to the primal objective as
\begin{equation}
\inf_{\pi \in \Pi(\mu,\nu)} \big \{\E[\|U-\tilde{Y}\|^2/2] + \varepsilon \KL{\pi}{\mu\otimes \nu} : (U,\tilde{X},\tilde{Y}) \sim \pi, \, \E[\tilde{X} \mid U]=\E[X] \big \},
\label{eq: evqr}
\end{equation}
where $\varepsilon > 0$ is a regularization parameter and $\mathsf{KL}$ denotes the Kullback-Leibler (KL) divergence (or relative entropy) defined by 
\[
\KL{P}{Q} := 
\begin{cases}
\int \log \frac{dP}{dQ}\, dP & \text{if $P \ll Q$}, \\
\infty & \text{otherwise}. 
\end{cases}
\]
We shall call (\ref{eq: evqr}) the \textit{entropic VQR} problem. Without the presence of covariates, entropic VQR reduces to standard entropic OT. The entropic VQR problem (\ref{eq: evqr}) admits a unique optimal solution under mild conditions.  

Entropic VQR was previously considered by \cite{carlier2022vector} for discrete marginals as a practical means to approximately solve the VQR problem (\ref{eq: vqr init 0}). However, to the authors' best knowledge, the fundamental structural and approximation properties of entropic VQR are still much less understood than entropic OT. The goal of this paper is to address some of these gaps. Our contributions are summarized as follows. First, we conduct an in-depth study of duality theory for entropic VQR. At least formally (e.g., by considering the fully discrete case), the dual problem for entropic VQR can be seen as 
\[
\sup_{(f,g,h)} \int f \,d\mu + \int h \, d\nu - \varepsilon \int_{\R^{d_y} \times \R^{d_x+d_y}} e^{ \frac{1}{\varepsilon} \left(f(u) +  \langle g(u), x\rangle + h(x,y) - \|u-y\|^2/2  \right) } \,d\mu(u) d\nu(x,y) + \varepsilon,
\]
where the supremum is taken over a suitable class of functions $f: \R^{d_y} \to \R, g: \R^{d_y} \to \R^{d_x}$, and $h: \R^{d_x+d_y} \to \R$.
As the first main result, we establish strong duality and dual attainment for entropic VQR when the supremum above is taken over $(f,g,h) \in L^1(\mu) \times L^1(\mu;\R^{d_x}) \times L^1(\nu)$. Importantly, our result allows for both $\mu$ and $\nu$ to have unbounded supports. The proof shows that the optimal coupling admits a density of the form 
\[
\frac{d\pi}{d(\mu \otimes \nu)}(u,x,y) = e^{\frac{1}{\varepsilon}\big(f(u)+\langle g(u),x\rangle +h(x,y)-\|u-y\|^2/2\big)},
\]
for suitable triplet of functions $(f,g,h)$, which  yields strong duality and dual attainment. Our proof adapts various techniques from duality theory for entropic OT that originate from \cite{csiszar1975divergence,follmer1988random,follmer1997entropy} (see also \cite{nutz2021introduction}). One obstacle in our proof is that (unconditional) moment constraints approximating the mean-independence constraint need not be continuous in total variation unless $X$ is compactly supported. We bypass this difficulty by imposing coercivity of the KL-divergence in the $1$-Wasserstein topology with respect to a modified Euclidean metric, which appears to be new and can be adapted to handle different conditional moment constraints. In addition, we show that the preceding coercivity assumption holds under a reasonably mild moment condition on $X$.
As another difficulty compared with standard entropic OT, one function $g$ appears in the dual problem through the interaction term, $\langle g(u),x \rangle$, between $u$ (the ``input" variable) and $x$ (part of the ``output'' variables). As such,  a more careful measure-theoretic argument is needed to separately construct dual potentials in a measurable way;  see the discussion above Lemma \ref{lem: measure}.  

Similar to entropic OT, dual potentials for entropic VQR (i.e., optimal solutions to the dual problem) are characterized by the system of functional equations that are akin to the Schr\"{o}dinger system. Our next goal is to study regularity of dual potentials via the said system of functional equations. In contrast to standard entropic OT, one dual potential $g$ is characterized only through an implicit fixed point equation involving other potentials, which poses a significant challenge to study regularity of dual potentials. To overcome the said obstacle, we employ theory of exponential families \cite{wainwright2008graphical,barndorff2014information} to establish real analyticity of dual potentials when $\mu,\nu$ are compactly supported.
The argument herein is new in the OT literature and may be of independent interest. 

Finally, as an important test case, we consider the setting where $\mu$ and $\nu$ are Gaussian. 
Building on our duality theory, we find a closed-form expression for the optimal coupling $\pi^\varepsilon$ for entropic VQR, which is again Gaussian. 
The weak limit $\pi^o$ of the optimal coupling $\pi^\varepsilon$ when $\varepsilon \to 0+$ solves the unregularized VQR problem (\ref{eq: vqr init 0}) (and indeed (\ref{eq: vqr indep}) because of Gaussianity), and we establish the precise approximation rate of $\pi^\varepsilon$ toward $\pi^o$ in the 2-Wasserstein distance.

\subsection{Related literature}
The literature related to this paper is broad, so we confine ourselves to references directly related to our work, other than those already discussed. 
We refer the reader to \cite{panaretos2020invitation,chewi2025statistical} as excellent reviews on statistical OT, which has seen extensive research activities in recent years.

VQR is a special case of weak OT with moment constraints considered in the recent preprint \cite{carlier2025weak}, which establishes strong duality and dual attainment in their general framework, including entropic variants. Importantly, however, \cite{carlier2025weak} require marginals to be compactly supported, and their proofs do not seem to carry over to the unbounded setting, at least directly. Their proof of strong duality rests on a Fenchel-Rockafellar argument, and their proof of dual attainment rests on taking a maximizing sequence for the dual objective and establishing a subsequential limit. These arguments largely differ from our proof. In particular, allowing for unbounded supports brings a major obstacle in our proof of strong duality and dual attainment and is needed to cover Gaussian marginals studied in the later section. Finally, regularity of dual potentials is not studied in \cite{carlier2025weak}.  As such, we view our work and \cite{carlier2025weak} as complementary. 

Another related work is \cite{beiglbock2025fundamental}, where the authors studied duality theory for weak OT under fairly general settings. Their results can be used to establish primal attainment for entropic VQR (see the proof of Proposition \ref{prop: primal attainment}), but our intended duality results do not seem to follow from their general results, at least directly, because it seems nontrivial to verify their Condition (B) in our context.

The structure of the mean-independence constraint in VQR is reminiscent of martingale OT considered in the mathematical finance literature (see, e.g., \cite{beiglbock2013model, galichon2014stochastic}), at least formally. For an entropic variant of martingale OT, \cite{nutz2024martingale} studied its duality theory in detail for two marginals defined on the real line. However, the setting and analysis of VQR differ substantially from martingale OT. In particular, in VQR, the mean-independence constraint is imposed on the auxiliary variable $\tilde{X}$ that does not enter the ground cost $\|U-\tilde{Y}\|^2$, and two marginals $\mu,\nu$ are defined on spaces with different dimensions. 

Gaussian distributions serve as an important test case for OT, as they often allow for closed-form expressions for OT costs and couplings; see \cite{janati2020entropic, mallasto2022entropy} for the case of standard entropic OT with quadratic cost.

\subsection{Organization}
The rest of the paper is organized as follows. Section \ref{sec: duality} formally sets up VQR and entropic VQR, establishes primal attainment,  and presents the results for duality theory. Section \ref{sec: gaussian} presents the results under Gaussian marginals.  Sections \ref{sec: proof duality} and \ref{sec: proof gaussian} collect proofs for Sections \ref{sec: duality} and \ref{sec: gaussian}, respectively. 

\subsection{Notation}
On a Euclidean space, let $\| \cdot \|$ and $\langle \cdot, \cdot \rangle$ denote the standard Euclidean norm and inner product, respectively. 
For any Polish metric space $(M,d)$, we use $\calB(M)$ to denote its Borel $\sigma$-field. Let $\calP(M)$ denote the collection of all Borel probability measures on $M$. When $M$ is a finite-dimensional Euclidean space, we consider the standard Euclidean metric, unless otherwise stated. For any $p \in [1,\infty)$ and any fixed $x_0 \in M$, let $\calP_p(M) := \{ \mu \in \calP(M): \int d^p(x,x_0) \, d\mu(x) < \infty \}$. For $\rho_0,\rho_1 \in \calP_p(M)$, the \textit{$p$-Wasserstein distance} is 
\[
\mathsf{W}_p(\rho_0,\rho_1) := \left ( \inf_{\pi \in \Pi(\rho_0,\rho_1)} \int d^p(x,y) \, d\pi(x,y) \right )^{1/p},
\]
which defines a metric on $\calP_p(M)$. 
For any $\mu \in \calP(M), p \in [1,\infty]$, and $d \in \bN$, let $L^p(\mu; \R^d)$ denote the $L^p(\mu)$-space of Borel measurable mappings $M \to \R^{d}$. We write $L^p(\mu) = L^p(\mu;\R)$. Finally,  for $a \in \R$, let $a^+=\max \{ a,0 \}$ and $a^- = \max \{ -a, 0 \}$; for $a,b \in \R$, we use the notation $a \wedge b = \min \{ a,b \}$.

\section{Duality theory for entropic VQR}
\label{sec: duality}
\subsection{VQR and entropic VQR}
We first fix notation. Let $(X,Y) \in \R^{d_x+d_y}$ be a pair of vectors of covariates $X \in \R^{d_x}$ and response variables $Y \in \R^{d_y}$. Let $\nu \in \calP(\R^{d_x+d_y})$ denote the joint distribution of $(X,Y)$, and let $\nu_x$ denote the conditional distribution of $Y$ given $X$. We assume throughout the paper that 
\[
\E\big[\|X\|^2 + \|Y\|^2\big] < \infty \quad \text{and} \quad \E[X]=0.
\]
The latter condition, $\E[X] = 0$, is for normalization and does not lose generality.
Let $\mu \in \calP(\R^{d_y})$ be a reference measure with a finite second moment. We do not assume that $\mu$ is absolutely continuous. 

For notational convenience, in what follows, we set
\[
c(u,y) := \frac{1}{2}\|u-y\|^2
\]
and write $\int c \, d\pi = \int_{\R^{d_y} \times \R^{d_x+d_y}} c(u,y) \, d\pi(u,x,y)$.
For any coupling $\pi \in \Pi(\mu,\nu)$, let $\pi_u$ denote the conditional distribution of $(\tilde X,\tilde Y)$ given $U$ when $(U,\tilde X,\tilde Y) \sim \pi$. That is, for any nonnegative measurable function $\varphi: \R^{d_y} \times \R^{d_x+d_y} \to [0,\infty]$, 
\begin{equation}
\int \varphi \, d\pi = \int_{\R^{d_y}}\left ( \int_{\R^{d_x+d_y}} \varphi (u,x,y) \, d\pi_u(x,y)\right ) \, d\mu(u). 
\label{eq: rcd}
\end{equation}
See Chapter 10.2 in \cite{Dudley_2002} for (regular) conditional distributions.
With this notation, the VQR problem (\ref{eq: vqr init 0}) reads as
\begin{equation}
\inf_{\pi \in \Pi(\mu, \nu)} \int c \, d\pi\ \  \text{subject to} \ \int_{\R^{d_x+d_y}} x \, d\pi_u(x,y) = 0 \ \text{$\mu$-a.e. $u$},
\label{eq: vqr}
\end{equation}
and the entropic VQR problem (\ref{eq: evqr}) reads as
\begin{equation} \label{eq: entropicPrimal}
\begin{split}
  \mathsf{T}^\varepsilon(\mu,\nu)&:=  \inf_{\pi \in \Pi(\mu, \nu)} \left( \int c \, d\pi + \varepsilon \KL{\pi}{\mu\otimes\nu} \right) \\
  &\qquad \text{subject to} \ \int_{\R^{d_x+d_y}} x \, d\pi_u(x,y) = 0 \ \text{$\mu$-a.e. $u$}.
  \end{split}
\end{equation}

Entropic VQR can be seen as the entropic projection of a modified reference measure onto the feasible set
\begin{equation}
\calQ := \left \{ \pi \in \Pi(\mu,\nu) : \int_{\R^{d_x + d_y}} x \, d\pi_u(x,y) =0 \ \text{$\mu$-a.e. $u$} \right \}.
\label{eq: feasible set}
\end{equation}
Indeed, defining 
\begin{equation}
d\tilde{R} := \alpha^{-1}e^{-c/\varepsilon} dR \quad \text{with} \quad \alpha := \int e^{-c/\varepsilon} \,dR \in (0,\infty), \label{eq: reference}
\end{equation}
we see that
\[
 \int c \, d\pi + \varepsilon \KL{\pi}{R}  = \varepsilon \KL{\pi}{\tilde{R}} - \varepsilon \log \alpha.
\]
Hence, the primal problem (\ref{eq: entropicPrimal}) is equivalent to minimizing $\KL{\pi}{\tilde{R}}$ over $\calQ$. Observe that the feasible set $\calQ$ is nonempty (as $\mu \otimes \nu \in \calQ$) and convex.

We first verify below that the entropic VQR problem (\ref{eq: entropicPrimal}) admits a unique optimal solution $\pi^\varepsilon$.

\begin{proposition}[Primal attainment]
\label{prop: primal attainment}
For every $\varepsilon > 0$, there exists a unique optimal solution $\pi^\varepsilon$ to the primal problem (\ref{eq: entropicPrimal}). 
\end{proposition}

In the rest of this section, we fix $\varepsilon > 0$ and study duality theory for entropic VQR.

\subsection{Duality}

Let $R:=\mu \otimes \nu$. Define the dual objective as
\[
D^\varepsilon(f,g,h) := \int f \,d\mu + \int h \, d\nu - \iota^\varepsilon (f,g,h)
\]
with
\[
\iota^\varepsilon (f,g,h) := \varepsilon \int_{\R^{d_y} \times \R^{d_x+d_y}} e^{ \frac{1}{\varepsilon} \left(f(u) +  \langle g(u), x\rangle + h(x,y) - c(u,y)  \right) } \,dR(u,x,y) - \varepsilon.
\]
The dual problem reads as
\begin{equation}
\mathsf{D}^\varepsilon (\mu,\nu) := \sup_{(f,g,h)} D^\varepsilon(f,g,h), 
\label{eq: dual}
\end{equation}
where the supremum is taken over $(f,g,h) \in L^1(\mu) \times L^1(\mu;\R^{d_x}) \times L^1(\nu)$.

We first establish weak duality. The proof is standard except for one place where we verify that $\langle g(u),x \rangle \in L^1(\pi)$ for any $\pi \in \calQ$ with $\KL{\pi}{R} < \infty$, provided that $\iota^\varepsilon (f,g,h) < \infty$. 

\begin{proposition}[Weak duality]
\label{prop: weak duality}
The weak duality holds: $\mathsf{T}^\varepsilon(\mu,\nu) \ge \mathsf{D}^\varepsilon(\mu,\nu)$. 
\end{proposition}

For strong duality and dual attainment, we will make the following additional assumption. Recall that $\tilde{R}$ is a probability measure on $\R^{d_y} \times \R^{d_x+d_y}$ defined in (\ref{eq: reference}). We consider the $1$-Wasserstein distance for a modified metric on $\R^{d_y} \times \R^{d_x+d_y}$ different from the Euclidean one. Define 
\[
\tilde{\mathsf{d}}\big((u,x,y),(u',x',y')\big) := \|u-u'\| \wedge 1 + \|x-x'\| + \| y-y' \| \wedge 1
\]
for $(u,x,y), (u',x',y') \in \R^{d_y} \times \R^{d_x+d_y}$. 
The metric $\tilde{\mathsf{d}}$ induces the same topology as the Euclidean one. 
We denote by $\tilde{\mathsf{W}}_1$ the corresponding $1$-Wasserstein distance,
\[
\tilde{\mathsf{W}}_1 (P,Q) := \inf_{\gamma \in \Pi(P,Q)} \int \tilde{\mathsf{d}} \, d\gamma,
\]
which defines a metric on
\begin{equation}
\tilde{\calP}_1 := \left  \{ P \in \calP(\R^{d_y} \times \R^{d_x+d_y}) : \int \tilde{\mathsf{d}} (\cdot,0) \, dP < \infty \right \}.
\label{eq: modified P1}
\end{equation}
By definition, $P \in \tilde{\calP}_1$ if and only if $\E[\|\tilde{X}\|] < \infty$ for $(U,\tilde{X},\tilde{Y}) \sim P$.
Observe that $P_n \to P$ in $\tilde{\mathsf{W}}_1$ if and only if, for any continuous function $\varphi: \R^{d_y} \times \R^{d_x+d_y} \to \R$ with $|\varphi(u,x,y)| \le K(1+\|x\|)$ for $(u,x,y) \in \R^{d_y} \times \R^{d_x+d_y}$ for some finite constant $K$, it holds that $\int \varphi \, dP_n \to \int \varphi dP$. See Theorem 6.9 in \cite{villani2009optimal}. In particular, the topology induced by $\tilde{\mathsf{W}}_1$ is stronger than the weak topology.
\begin{assumption}
\label{asmp: dual}
(i) The matrix $\E[XX^\top]$ is invertible. 
(ii) The mapping $Q \mapsto \mathsf{KL}(Q \, \| \, \tilde{R})$ is coercive in $\tilde{\mathsf{W}}_1$, i.e., for any $0 < a < \infty$, the sublevel set
\[
\big\{ Q \in \tilde{\calP}_1 : \mathsf{KL}(Q \, \| \, \tilde{R}) \le a \big\}
\]
is compact for the $\tilde{\mathsf{W}}_1$-topology. 
\end{assumption}

Condition (i) guarantees that $X$ is not concentrated on any affine hyperplane in $\R^{d_x}$. The function $g$ enters the dual objective $D(f,g,h)$ only through $\langle g(u),x \rangle$, and the said condition guarantees to recover $g$ from $\langle g(u),x \rangle$. 
Condition (ii) is a high-level condition that guarantees the existence of the entropic projection onto the set of probability measures satisfying a finite number of unconditional moment constraints that approximate the feasible set $\calQ$ in (\ref{eq: feasible set}).
Since the support of $X$ may be unbounded, these unconditional moment constraints need not be continuous in total variation, but continuous in $\tilde{\mathsf{W}}_1$; see Lemma \ref{lem: moment condition} below. Coercivity of $\mathsf{KL}(\cdot \, || \, \tilde{R})$ in $\tilde{\mathsf{W}}_1$ ensures the existence of the said entropic projection.
Condition (ii) is satisfied under a suitable moment condition on $X$.

\begin{lemma}
\label{lem: coercive X}
Assumption \ref{asmp: dual}~(ii) holds if
\begin{equation}
\E\big [ e^{\alpha' \| X \|}\big] < \infty, \ \forall \alpha' > 0. 
\label{eq: superexponential}
\end{equation}
\end{lemma}
\begin{remark} A few remarks are in order. 
\begin{enumerate}
\item[(i)] Recall that a real-valued random variable $\xi$ or its law is called \textit{$\beta$-sub-Weibull} for some $\beta \in (0,\infty)$ if 
\[
\| \xi \|_{\psi_{\beta}} := \inf \left\{ K>0 : \E\big[e^{|\xi/K|^\beta} \big] \le 2 \right\} < \infty.
\]
We say that a random vector $Z$ or its law is $\beta$-sub-Weibull if $\| \| Z \| \|_{\psi_\beta} < \infty$. 
Condition (\ref{eq: superexponential}) is satisfied if $X$ is $\beta$-sub-Weibull for some $\beta > 1$. 
\item[(ii)] Condition (\ref{eq: superexponential}) is not necessary for Assumption \ref{asmp: dual}~(ii) to hold. See Remark \ref{rem: necessity} below. 
\end{enumerate}
\end{remark}

Now, we state the first main theorem of this section.

\begin{theorem}[Strong duality and dual attainment]\label{thm: duality}
Under Assumption \ref{asmp: dual}, the following hold. 
\begin{enumerate}
\item[(i)] (Strong duality). $\mathsf{T}^\varepsilon(\mu,\nu) = \mathsf{D}^\varepsilon(\mu,\nu)$. 

\item[(ii)] (Dual attainment). There exist functions $(f^\varepsilon, g^\varepsilon, h^\varepsilon) \in L^1(\mu) \times L^1(\mu;\R^{d_x}) \times L^1(\nu)$ such that 
\begin{equation}
\frac{d\pi^{\varepsilon}}{dR} (u,x,y) = e^{\frac{1}{\varepsilon}(f^{\varepsilon}(u) + \langle g^\varepsilon(u), x\rangle  +h^{\varepsilon}(x,y) -c(u,y))}.
\label{eq: density}
\end{equation}
These functions are optimal for the dual problem (\ref{eq: dual}). Finally, the primal cost $\mathsf{T}^\varepsilon(\mu,\nu)$ is expressed as
\[
\mathsf{T}^\varepsilon(\mu,\nu) = \int f^\varepsilon \, d\mu + \int h^\varepsilon \, d\nu. 
\]
\end{enumerate}
\end{theorem}

We shall call any triplet of functions $(f,g,h)$ attaining the supremum in the dual problem (\ref{eq: dual}) \textit{dual potentials}. With respect to the modified reference measure $\tilde{R}$, against which we take the entropic projection, the optimal coupling $\pi^\varepsilon$ has a density of the form
\[
\frac{d\pi^{\varepsilon}}{d\tilde{R}} (u,x,y) = e^{\frac{1}{\varepsilon}(f^{\varepsilon}(u) + \langle g^\varepsilon(u), x\rangle  +h^{\varepsilon}(x,y)) + \log \alpha}.
\]
In contrast to standard entropic OT, there is an interaction term between $u$ and $x$ via $\langle g^\varepsilon(u),x \rangle$ due to the mean-independence constraint, so the log density does not factor into a tensor sum of two separate functions of $u$ and $(x,y)$ only.

The next proposition addresses uniqueness of dual potentials. It shows that dual potentials are unique up to an affine shift.

\begin{proposition}[Uniqueness of dual potentials]
\label{prop: uniqueness}
Suppose that Assumption \ref{asmp: dual}~(i) holds. If $(f,g,h), (\tilde f,\tilde g,\tilde h) \in L^1(\mu) \times L^1(\mu;\R^{d_x}) \times L^1(\nu)$ are both optimal solutions to (\ref{eq: dual}), then there exist $a \in \R$ and $v \in \R^{d_x}$ such that
\[
\begin{split}
&\tilde f (u)=f (u)+ a, \ \tilde g (u)= g (u)+ v, \ \text{$\mu$-a.e. $u$,} \\
&\tilde h (x,y)= h (x,y)- a - \langle v,x\rangle, \ \text{$\nu$-a.e. $(x,y)$}.
\end{split}
\]
\end{proposition}

The following proposition is a converse to Theorem \ref{thm: duality}~(ii). 

\begin{proposition}
\label{prop: converse}
Suppose that $\pi$ is feasible for the primal problem (\ref{eq: entropicPrimal}) (i.e., $\pi \in \calQ$)  and of the form 
\begin{equation}
\frac{d\pi}{dR}(u,x,y) = e^{\frac{1}{\varepsilon}(f(u)+\langle g(u),x\rangle+h(x,y)-c(u,y))}
\label{eq: exponential}
\end{equation}
for some $(f,g,h) \in L^1(\mu) \times L^1(\mu;\R^{d_x}) \times L^1(\nu)$. Then $\pi$ is optimal for (\ref{eq: entropicPrimal}) and hence $\pi = \pi^\varepsilon$. 
\end{proposition}

Combining Theorem \ref{thm: duality}~(ii) and the preceding proposition, we obtain the following characterization of dual potentials, which is akin to the Schr\"{o}dinger system in entropic OT.
\begin{corollary}
Suppose that Assumption \ref{asmp: dual} holds.
For given $(f,g,h) \in L^1(\mu) \times L^1(\mu;\R^{d_x}) \times L^1(\nu)$, they solve the dual problem (\ref{eq: dual}) if and only if they satisfy
\begin{align}
&f(u) = -\varepsilon \log \int_{\R^{d_x+d_y}} e^{\frac{1}{\varepsilon}(\langle g(u), x\rangle + h(x,y) -c(u,y))} \,d\nu(x,y)\quad \text{$\mu$-a.e. $u$}, \label{eq: f}\\
&\int_{\R^{d_x+d_y}} x \cdot e^{\frac{1}{\varepsilon}(\langle g(u), x\rangle + h(x,y) -c(u,y))} d\nu(x,y) = 0 \quad \text{$\mu$-a.e. $u$}, \label{eq: g} \\
&
h(x,y) = -\varepsilon \log \int_{\R^{d_y}} e^{\frac{1}{\varepsilon}(f(u) + \langle g(u), x\rangle  -c(u,y))}\, d\mu(u) \quad \text{$\nu$-a.e. $(x,y)$.} \label{eq: h}
\end{align}
\end{corollary}

Indeed, let $\pi$ be a Borel measure on $\R^{d_y} \times \R^{d_x+d_y}$ of the form (\ref{eq: exponential}). 
Equations (\ref{eq: f}) and (\ref{eq: h}) are equivalent to $\pi$ being a coupling for $(\mu,\nu)$. Equation (\ref{eq: g}) is equivalent to $\pi$ satisfying the mean-independence constraint in (\ref{eq: entropicPrimal}). Hence, the conclusion follows from Theorem \ref{thm: duality}~(ii) and Proposition \ref{prop: converse}.

\begin{remark}
\label{rem: Schrodinger}
One can choose versions of dual potentials $(f^\varepsilon,g^\varepsilon,h^\varepsilon)$ so that (\ref{eq: f}) and (\ref{eq: h}) hold for \textit{all} $u \in \R^{d_y}$ and $(x,y) \in \R^{d_x+d_y}$, respectively. Indeed, for given dual potentials $(f^\varepsilon,g^\varepsilon,h^\varepsilon)$, define $\tilde{f}^\varepsilon$ and $\tilde{h}^\varepsilon$ by the right-hand sides on (\ref{eq: f}) and (\ref{eq: h}), respectively, then $\tilde{f}^\varepsilon = f^\varepsilon$ $\mu$-a.e. and $\tilde{h}^\varepsilon = h^\varepsilon$ $\nu$-a.e. Both $\tilde{f}^\varepsilon$ and $\tilde{h}^\varepsilon$ may take $-\infty$, but they are integrable under $\mu$ and $\nu$, respectively, and are finite $\mu$-a.e. and $\nu$-a.e., respectively. 
However, it is unclear whether one can choose a version of $g^\varepsilon$ so that (\ref{eq: g}) holds for all $u \in \R^{d_y}$.
Proposition \ref{prop: exponential} below shows that one can choose such a version of $g^\varepsilon$ at least when $\mu$ and $\nu$ are compactly supported. 
\end{remark}

\subsection{Regularity of dual potentials}
\label{sec: regularity}

In this section, we assume the existence of dual potentials satisfying the Schr\"{o}dinger-like system (\ref{eq: f})--(\ref{eq: h}) and study their regularity. Throughout this section, we choose versions of dual potentials $(f^\varepsilon,g^\varepsilon,h^\varepsilon)$ so that (\ref{eq: f}) and (\ref{eq: h}) hold for \textit{all} $u \in \R^{d_y}$ and $(x,y) \in \R^{d_x+d_y}$, respectively; see Remark \ref{rem: Schrodinger}. Furthermore, let $\calU \subset \R^{d_y}$ denote the support of $\mu$, and let $\calX \subset \R^{d_x},\calY \subset \R^{d_y}$ denote the supports of $X,Y$, respectively. We will assume throughout this section that $\calU$ is compact.  

The first two propositions concern regularity of $h^\varepsilon$. 
\begin{proposition}
\label{prop: continuity}
Suppose that $\calU$ is compact. Pick any $x \in \R^{d_x}$ and $y_1,y_2 \in \R^{d_y}$. If at least one of $h^\varepsilon(x, y_1)$ or $h^\varepsilon(x, y_2)$ is finite, then both are finite and 
\[
|h^\varepsilon(x, y_1) - h^\varepsilon(x, y_2)| \le \sup_{u\in \calU }| c(u,y_1) - c(u,y_2)|.
\]
\end{proposition}

\begin{remark}
Since $h^\varepsilon$ is finite $\nu$-a.e., with $\kappa$ denoting the distribution of $X$, the proposition implies that for $\kappa$-a.e. $x$, $h^\varepsilon(x,\cdot)$ is everywhere finite and locally
Lipschitz. 
\end{remark}

\begin{proposition}
\label{prop: convexity}
Suppose that Assumption \ref{asmp: dual}~(i) holds and that $\calU$ is compact. 
For any fixed $y \in \R^{d_y}$, $h^\varepsilon(x, y)$ is concave in $x$.
Furthermore, the function $x\mapsto h^\varepsilon(x, y)$ is finite and locally Lipschitz on the interior of the convex hull of $\calX$. 
\end{proposition}

For the rest of this section, we will assume that $\calX$ and $\calY$ are both compact, in addition to compactness of $\calU$. The next proposition shows that one can choose a version of $g^\varepsilon$ so that (\ref{eq: g}) holds for all $u$. The proof relies on theory of exponential families (cf. Chapter 3 of \cite{wainwright2008graphical} and \cite{barndorff2014information}).

\begin{proposition}
\label{prop: exponential}
Suppose that Assumption \ref{asmp: dual}~(i) holds  and that $\calU, \calX$, and $\calY$ are compact. 
For every $u \in \R^{d_y}$, there exists a unique  $\theta \in \R^{d_x}$ such that 
\begin{equation}
\int x \cdot \exp\left(\frac{ \langle \theta,x \rangle+ h^\varepsilon(x,y) - c(u,y)}{\varepsilon}\right) d\nu(x,y) = 0.
\label{eq: implicit}
\end{equation}
\end{proposition}

The preceding proposition guarantees that, given $h^\varepsilon$, one can choose a version of $g^\varepsilon$ such that it is everywhere finite and  (\ref{eq: g}) holds for \textit{all} $u \in \R^{d_y}$. 
It is not difficult to verify that $f^\varepsilon$ and $h^\varepsilon$ are everywhere finite. Furthermore, the next theorem establishes that $(f^\varepsilon,g^\varepsilon,h^\varepsilon)$ are real analytic. 
We refer the reader to \cite{krantz2002primer} as an excellent reference on theory of real analytic functions.

\begin{theorem}[Real analyticity of dual potentials]
\label{thm: analytic}
Suppose that Assumption \ref{asmp: dual}~(i) holds  and that $\calU, \calX$, and $\calY$ are compact.  Then: (i) $f^\varepsilon$ and $g^\varepsilon$ are real analytic functions on $\R^{d_y}$ (more precisely, each coordinate of $g^\varepsilon$ is real analytic), and (ii) $h^\varepsilon$ is a real analytic function on $\R^{d_x+d_y}$.
\end{theorem}

Arguably, the most challenging part is the analysis of $g^{\varepsilon}$. Assume $\varepsilon = 1$ and drop the superscript $\varepsilon$ for simplicity. In view of Proposition \ref{prop: exponential}, $g(u)$ is the unique vector in $\R^{d_x}$ such that $\nabla_\theta A(g(u),u) = 0$, where 
\[
A(\theta,u) = \int e^{\langle \theta,x\rangle + \langle u,y \rangle } \, d\tilde{\nu}(x,y) \quad \text{with} \quad d\tilde{\nu}(x,y) = e^{h(x,y)-\|y\|^2/2} \, d\nu(x,y).
\]
Since $A(\theta,u)$ is the Laplace transform of the base measure $\tilde{\nu}$, it is analytic by \cite[Theorem 2.7.1]{barndorff2014information}. Real analyticity of $g$ follows from applying the real analytic version of the implicit function theorem \cite[Theorem 2.3.5]{krantz2002primer}.

\section{Gaussian case}
\label{sec: gaussian}

\subsection{Closed-form solution}

In this section, we derive closed-form solutions for the entropic VQR problem under Gaussian marginals.

\begin{theorem}
\label{thm: gaussian solution}
Suppose that $\mu = \mathcal{N}(0, I_{d_y})$  and $\nu = \mathcal{N}\big((0^\top, m_Y^\top)^\top, \Sigma\big)$, where $m_Y := \E[Y]$ and the covariance matrix $\Sigma$ is partitioned as
\[
\Sigma = \begin{pmatrix} \Sigma_{XX} & \Sigma_{XY} \\ \Sigma_{YX} & \Sigma_{YY} \end{pmatrix}.
\]
Assume that $\Sigma$ is nonsingular. 
Then the following hold.

\begin{enumerate}
\item[(i)] The optimal coupling $\pi^\varepsilon$ for the entropic VQR problem (\ref{eq: entropicPrimal}) is a nondegenerate multivariate Gaussian distribution $\mathcal{N}(m, \Gamma_\varepsilon)$ with
\[
m := \begin{pmatrix}
0 \\ 0 \\ m_Y
\end{pmatrix} \quad\text{and}\quad \Gamma_\varepsilon := \begin{pmatrix} 
I_{d_y} & O & \Lambda_{\varepsilon} \\ O & \Sigma_{XX} & \Sigma_{XY} \\ \Lambda_{\varepsilon} & \Sigma_{YX} & \Sigma_{YY} \end{pmatrix},
\]
where $\Lambda_{\varepsilon}$ is the $d_y \times d_y$ symmetric positive definite matrix given by
\[
\Lambda_{\varepsilon} := \left(\Sigma_{YY} - \Sigma_{YX}\Sigma_{XX}^{-1}\Sigma_{XY} + \frac{\varepsilon^2}{4}I_{d_y}\right)^{1/2} - \frac{\varepsilon}{2}I_{d_y}.
 \]
\item[(ii)] One can choose versions of dual potentials as
 \begin{equation}
 \begin{aligned}
 f^\varepsilon(u) &= -\frac{1}{2}u^\top (\Lambda_{\varepsilon} - I_{d_y})u - m_Y^\top u - \frac{\varepsilon}{2} \log \det(\varepsilon \Lambda_{\varepsilon} \Omega_{YY}), \\
 g^\varepsilon(u) &= G u, \\
h^\varepsilon(x,y) &= -\frac{1}{2}(G^\top x + y - m_Y)^\top \Psi_\varepsilon (G^\top x + y - m_Y) + \frac{1}{2}\|y\|^2 ,
\end{aligned}
\label{eq: gaussian dual potentials}
\end{equation}
where $G := -\Sigma_{XX}^{-1}\Sigma_{XY}$, $\Psi_\varepsilon := \Lambda_{\varepsilon} \Omega_{YY}$, and $\Omega_{YY} := (\Sigma_{YY} - \Sigma_{YX}\Sigma_{XX}^{-1}\Sigma_{XY})^{-1}$. 
\end{enumerate}
\end{theorem}

In the absence of covariates $X$ (in which case entropic VQR reduces to standard entropic OT with quadratic cost), our results are consistent with those of \cite{janati2020entropic}. The proof uses a key observation from our duality theory (note that Gaussian distributions satisfy Assumption \ref{asmp: dual}~(ii) by Lemma \ref{lem: coercive X}) that, the cross term between $u$ and $y$ in the log (Lebesgue) density of $\pi^\varepsilon$ is $u^\top y/\varepsilon$, which implies that the corresponding block in the precision matrix $\Gamma_\varepsilon^{-1}$ is proportional to $I_{d_y}$. Using this, we derive a matrix Riccati-type equation for $\Lambda_\varepsilon$ (see (\ref{eq: Riccati}) below), solving which yields (i). 
Part (ii) follows by comparing two density expressions.

\subsection{Convergence of entropic VQR}

In this section, using the closed-form expression of the optimal coupling  $\pi^{\varepsilon}$, we study convergence of $\pi^\varepsilon$ when $\varepsilon \to 0+$. The next result is straightforward. 

\begin{proposition}
\label{prop: gaussian limit}
Consider the setting of Theorem \ref{thm: gaussian solution}. 
As $\varepsilon \to 0+$, we have $\pi^\varepsilon \to \pi^o$ weakly, where $\pi^o = \mathcal{N}(m,\Gamma_o)$ with
\[
\Gamma_o :=
\begin{pmatrix} 
I_{d_y} & O & \Lambda_{o} \\ O & \Sigma_{XX} & \Sigma_{XY} \\ \Lambda_{o} & \Sigma_{YX} & \Sigma_{YY} \end{pmatrix} 
\quad \text{and} \quad \Lambda_o := (\Sigma_{YY}-\Sigma_{YX}\Sigma_{XX}^{-1}\Sigma_{XY})^{1/2}.
\]
The limiting law $\pi^o$ is optimal for the unregularized VQR problem (\ref{eq: vqr indep}) with the independence constraint.
For $(U,\tilde{X},\tilde{Y}) \sim \pi^o$, one has $\tilde{Y} = m_Y + \Lambda_o U + \Sigma_{YX}\Sigma_{XX}^{-1}\tilde{X}$ a.s.
\end{proposition}

Furthermore, we obtain the asymptotic expansion of the $2$-Wasserstein distance between $\pi^\varepsilon$ and $\pi^o$. 
\begin{proposition}
\label{prop: gaussian approx}
Under the setting of Theorem \ref{thm: gaussian solution}, 
we have as $\varepsilon \to 0+$,
\[
\mathsf{W}_2^2(\pi^\varepsilon,\pi^o) = \varepsilon \tr (L^{-1}\Lambda_o)+ O(\varepsilon^2) \quad \text{with} \quad L:= I_{d_y} + \Lambda_o^2 + \Sigma_{YX}\Sigma_{XX}^{-2}\Sigma_{XY}.
\]
\end{proposition}
The proposition implies that the precise order of $\mathsf{W}_2^2(\pi^\varepsilon,\pi^o)$ is $\varepsilon$ under Gaussian marginals. It is worth pointing out that, for standard entropic OT, the recent paper by \cite{malamut2025convergence} establishes that the approximation rate in $\mathsf{W}_2^2$ of the optimal entropic coupling under quadratic cost is precisely $\varepsilon$, for absolutely continuous marginals that admit a Lipschitz continuous Brenier map; see Theorem 3.9 in \cite{malamut2025convergence}. The above result is consistent with theirs.

\section{Proofs for Section  \ref{sec: duality}}
\label{sec: proof duality}

We recall some properties of the KL divergence. Let $M$ be a Polish metric space. For $P,Q \in \calP(M)$, $\KL{P}{Q} \in [0,\infty]$ and $\KL{P}{Q}=0$ if and only if $P=Q$. In addition, the mapping $P \mapsto \KL{P}{Q}$ is convex, lower semicontinuous for the weak topology (and any stronger topology), and strictly convex on its domain. The lower semicontinuity follows by the Donsker-Varadhan variational representation (cf. Theorem 4.6 in \cite{polyanskiy2025information}).

\subsection{Proof of Proposition \ref{prop: primal attainment}}
Uniqueness follows from strict convexity of $\KL{\cdot}{\mu\otimes \nu}$ on its domain, so we prove the existence. In fact, it is not very hard to verify that $\calQ$ is weakly compact (as each $\pi \in \calQ$ has fixed marginals and $X$ has finite expectation), which yields the existence.  

We provide an alternative argument by rewriting
 the primal problem (\ref{eq: entropicPrimal}) as a weak OT problem \cite{beiglbock2025fundamental}. 
Since
$
\KL{\pi}{\mu \otimes \nu} = \int_{\R^{d_y}} \KL{\pi_u}{\nu}\, d\mu(u)
$ \cite[Theorem 2.15]{polyanskiy2025information},
we have
\[
\int c \, d\pi + \varepsilon \KL{\pi}{\mu\otimes\nu} = \int_{\R^{d_y}} \left(\int_{\R^{d_x+d_y}} c(u,y) \, d\pi_u(x,y) + \varepsilon \KL{\pi_u}{\nu} \right)\, d\mu(u).
\]
Let
\[
\calC:= \left \{\rho \in \mathcal{P}_2(\R^{d_x+d_y}): \int_{\R^{d_x+d_y}} x\, d\rho(x,y) = 0 \right\},
\]
and define the convex indicator
\[
\chi_{\calC}(\rho) := \begin{cases}0, & \text{if $\rho\in \mathcal{C}$,} \\
    \infty, & \text{otherwise.}
    \end{cases}
\]
The mean-independence constraint in (\ref{eq: entropicPrimal}) is equivalent to
$
\int \chi_{\calC}(\pi_u) \, d\mu(u) = 0. 
$
Hence, the primal problem (\ref{eq: entropicPrimal}) is equivalent to the following weak OT problem:
\begin{equation}
\label{weakOTprimal}
\inf_{\pi \in \Pi(\mu, \nu)} \int_{\R^{d_y}} C(u, \pi_u) \, d\mu(u), 
\end{equation}
where the cost $C: \R^{d_y} \times \calP_2(\R^{d_x+d_y}) \to [0,\infty]$ is given by
\begin{equation}
C(u, \rho) := \int_{\R^{d_x+d_y}} c(u,y) \, d\rho(x,y) + \varepsilon \KL{\rho}{\nu} + \chi_\calC(\rho).
\label{eq: weak cost}
\end{equation}
We will apply Theorem 2.2~(i) in \cite{beiglbock2025fundamental} to establish the desired claim. Endow $\calP_2(\R^{d_x+d_y})$ with the $2$-Wasserstein distance. We need to verify that the cost $C$ is Borel, and that $\rho \mapsto C(u,\rho)$ is convex and lower semicontinuous. The mapping $(u,\rho) \mapsto \int_{\R^{d_x+d_y}} c(u,y) \, d\rho(x,y)$ is jointly continuous and linear in $\rho$, and $\rho \mapsto \KL{\rho}{\nu}$ is convex and lower semicontinuous. The set $\calC$ is convex and closed in $\calP_2(\R^{d_x+d_y})$, so $\chi_\calC$ is convex and lower semicontinuous. As such, applying  Theorem 2.2~(i) in \cite{beiglbock2025fundamental}, we obtain the desired claim. \qed

\begin{remark}
As discussed in the introduction, it seems highly nontrivial to verify Condition (B) in \cite{beiglbock2025fundamental} for the cost (\ref{eq: weak cost}), and as such, it is unclear whether their Theorem 1.2 leads to our intended duality results.
\end{remark}
\subsection{Proof of Proposition \ref{prop: weak duality}}

It suffices to show that $\mathsf{T}^\varepsilon(\mu,\nu) \ge D^\varepsilon(f,g,h)$ for all $(f,g,h) \in L^1(\mu) \times L^1(\mu;\R^{d_x}) \times L^1(\nu)$. We may assume that $\iota^\varepsilon(f, g, h) < \infty$; otherwise, $D(f, g, h) = -\infty$ and the conclusion trivially holds.

Pick any feasible coupling $\pi \in \calQ$ with  $\KL{\pi}{R} < \infty$, and let $p_\pi:= \frac{d\pi}{dR}$. Then $\int p_\pi \, dR = \int d\pi = 1$. Let $F(u,x,y):= f(u) + \langle g(u), x\rangle + h(x,y)$. Observe that 
\[
\int e^{\frac{1}{\varepsilon}(F - c)}\,dR = \frac{1}{\varepsilon} \iota^\varepsilon(f, g, h) + 1 < \infty.
\] 
Apply the inequality
\[
\alpha\log\alpha - \alpha \ge \beta \alpha -e^{\beta}, \ \alpha \ge 0, \beta \in \R, 
\]
which follows from Fenchel's inequality,
with $\alpha = p_\pi$ and $\beta = (F-c)/\varepsilon$, to obtain
\begin{equation}
\varepsilon (p_\pi \log p_\pi - p_\pi) \ge (F-c)p_\pi - \varepsilon e^{\frac{1}{\varepsilon}(F-c)}.
\label{eq: fenchel}
\end{equation}

We shall verify that $\langle g(u),x \rangle \in L^1(\pi)$. Since the left-hand side and the second term on the right-hand side of (\ref{eq: fenchel}) are integrable under $R$, we have $(F-c)^+p_\pi \in L^1(R)$, i.e., $(F-c)^+ \in L^1(\pi)$. Observe that
\[
(\langle g(u),x \rangle)^+ \le (F(u,x,y)-c(u,y))^{+} + (f(u)+h(x,y)-c(u,y))^-.
\]
The right-hand side is integrable under $\pi$, so that $(\langle g(u),x \rangle)^+ \in L^1(\pi)$. Now, since $\pi$ is a coupling for $(\mu,\nu)$, $\| x \|$ is integrable under $\pi$, so that (cf. (\ref{eq: rcd}))
    \[
    \int_{\R^{d_x+d_y}} \|x\| \, d\pi_u(x,y) < \infty \ \text{$\mu$-a.e. $u$.}
    \]
 Pick any $u$ for which the integral on the left-hand side is finite. 
   Then
    \[
    \int_{\R^{d_x+d_y}} \langle g(u),x\rangle \, d\pi_u(x,y) = g(u)^\top \int_{\R^{d_x+d_y}} x \, d\pi_u(x,y) = 0.
    \]
    This implies that 
    \[
    \int_{\R^{d_x+d_y}} (\langle g(u),x\rangle)^- \, d\pi_u(x,y) = \int_{\R^{d_x+d_y}} (\langle g(u),x\rangle)^+ \, d\pi_u(x,y).
    \]
    Hence,  
    \[
    \begin{split}
    \int_{\R^{d_y} \times \R^{d_x+d_y}}  (\langle g(u),x\rangle)^{-} \, d\pi(u,x,y) &= \int_{\R^{d_y}} \int_{\R^{d_x+d_y}} (\langle g(u),x\rangle)^- \, d\pi_u(x,y) \, d\mu(u) \\
    &= \int_{\R^{d_y}} \int_{\R^{d_x+d_y}} (\langle g(u),x\rangle)^+ \, d\pi_u(x,y) \, d\mu(u) \\
    &< \infty.
    \end{split}
    \]
    Conclude that $\langle g(u),x\rangle \in L^1(\pi)$, which ensures that, for $(U,\tilde{X},\tilde{Y}) \sim \pi$, 
    \[
   \int_{\R^{d_y} \times \R^{d_x+d_y}} \langle g(u),x \rangle \, d\pi(u,x,y) = \E\big[ \langle g(U),\tilde{X} \rangle \big] = \E\big[\langle g(U),\E[\tilde{X} \mid U] \rangle \big] = 0.
   \]

Finally, using inequality (\ref{eq: fenchel}) again, we obtain
\[
\varepsilon\int (p_\pi\log p_\pi - p_\pi)\,dR \ge \int (F - c)p_\pi\, dR - \varepsilon \int e^{\frac{1}{\varepsilon}(F - c)}\,dR,
\]
that is, 
\[
\varepsilon \KL{\pi}{R} - \varepsilon \ge  \left( \int f \,d\mu + \int h \,d\nu - \int c \,d\pi \right) - \iota^\varepsilon(f, g, h) - \varepsilon.
\]
Rearranging terms, we obtain the desired claim. \qed

\subsection{Proof of Lemma \ref{lem: coercive X}}

Lemma \ref{lem: coercive X} directly follows from the next lemma and the fact that $e^{-c/\varepsilon}$ is bounded.

\begin{lemma}
\label{lem: coercive}
Let $(M,d)$ be a Polish metric space, and let $Q \in \calP(M)$ be satisfy 
\begin{equation}
\int e^{\alpha d(x,x_0)} \, dQ(x) < \infty, \ \forall \alpha > 0,
\label{eq: subexponential}
\end{equation}
for some $x_0 \in M$. 
Then the mapping $P \mapsto \KL{P}{Q}$ is coercive in $\mathsf{W}_1$.
\end{lemma}

There is a partial converse to Lemma \ref{lem: coercive}; see Appendix \ref{sec: converse}.

\begin{proof}[Proof of Lemma \ref{lem: coercive}]

The lemma (indirectly) follows from Theorem 2.1 in \cite{wang2010sanov} that establishes the $\mathsf{W}_p$-version of Sanov's theorem on large deviations for any $p \in [1,\infty)$. We provide a more direct proof using the weighted Pinsker inequality of Bolley and Villani \cite{bolley2005weighted}, which may be of independent interest.

Fix any $0<a < \infty$ and consider the sublevel set
\[
\calL_a = \{ P \in \calP_1(M): \KL{P}{Q} \le a \}.
\]
We first show that $\calL_a$ is weakly precompact.\footnote{Weak precompactness of $\calL_a$ also follows from the (classical) Sanov theorem (cf. Theorem 3.2.17  in \cite{deuschel2001large}).} 
Pick any $P \in \calL_a$. By the data processing inequality (cf. Theorem 7.4 in   \cite{polyanskiy2025information}), one has 
\[
\KL{P}{Q} \ge \mathsf{kl}(P(A) \, \| \, Q(A)), \ A \subset M,
\]
where 
\[
\mathsf{kl}(p  \, \| \, q) = p\log \frac{p}{q} + (1-p)\log \frac{1-p}{1-q}
\]
is the binary relative entropy.  Observe that for $p \in [0,1]$ and $q \in (0,1)$,
\[
\mathsf{kl}(p  \, \| \, q) = p\log \frac{1}{q} + \underbrace{\big [ p \log p + (1-p)\log (1-p) \big]}_{\ge -\log 2}+ \underbrace{(1-p)\log \frac{1}{1-q}}_{\ge 0} \ge p\log \frac{1}{q} - \log 2.
\]
Hence,
\[
P(A) \le \frac{a + \log 2}{\log (1/Q(A))},
\]
which implies that $\calL_a$ is uniformly tight and hence weakly precompact by Prohorov's theorem.

To show that $\calL_a$ is compact in $\mathsf{W}_1$, it remains to establish that 
\begin{equation}
\lim_{\lambda \to \infty} \sup_{P \in \calL_a} \int d(\cdot,x_0) \mathbbm{1}_{\{ d(\cdot,x_0) > \lambda \}} \, dP  = 0.
\label{eq: UI}
\end{equation}
See Theorem 6.9 in \cite{villani2009optimal}. Set $\varphi_\lambda := d(\cdot,x_0) \mathbbm{1}_{\{ d(\cdot,x_0) > \lambda \}}$. By the weighted Pinsker inequality from Theorem 2.1 in \cite{bolley2005weighted}, we have 
\[
\begin{split}
\int  \varphi_\lambda  \, dP &\le \int \varphi_\lambda \, dQ + \frac{2}{\alpha} \left (\frac{3}{2} + \log \int e^{\alpha \varphi_\lambda} \, dQ \right ) \left ( \sqrt{\KL{P}{Q}} + \frac{1}{2}\KL{P}{Q} \right ) \\
&\le \int \varphi_\lambda \, dQ +\frac{2}{\alpha} \left (\frac{3}{2} + \log \int e^{\alpha \varphi_\lambda} \, dQ \right ) \left ( \sqrt{a} + \frac{a}{2} \right), \quad \forall P \in \calL_a. 
\end{split}
\]
Since $d(\cdot,x_0) \in L^1(Q)$ by Condition (\ref{eq: subexponential}),  the dominated convergence theorem yields
\[
\lim_{\lambda \to \infty} \int \varphi_\lambda \, dQ  = 0.
\]
In addition,
\[
\int e^{\alpha \varphi_\lambda} \, dQ  = Q(\{ d(\cdot,x_0) \le \lambda \}) + \int_{\{ d(\cdot,x_0) > \lambda\}} e^{\alpha d(\cdot,x_0)} \, dQ.
\]
Again, by the dominated convergence theorem, the right-hand side converges to 1 as $\lambda \to \infty$.
By taking the supremum over $P \in \calL_a$ and taking the limits $\lambda \to \infty$ and  $\alpha \to \infty$ in this order, we obtain (\ref{eq: UI}).
\end{proof}

\begin{remark}[Necessity of Condition \ref{eq: superexponential} in Lemma \ref{lem: coercive X}]
\label{rem: necessity}
Assumption \ref{asmp: dual}~(ii) may hold even when Condition \ref{eq: superexponential} is not met. To see this, observe that, for $\alpha' > 0$,
\[
\int e^{\alpha' \tilde{\mathsf{d}}(\cdot,0)} \, d\tilde{R} \le \alpha^{-1} e^{2\alpha'} \int e^{\alpha' \|x\| - c(u,y)/\varepsilon} \, d\mu(u) d\nu(x,y).
\]
Since $c(u,y)\ge -\|u\|^2/2 + \|y\|^2/4$, we have
\[
\int e^{\alpha' \|x\| - c(u,y)/\varepsilon} \, d\mu(u) d\nu(x,y) \le \underbrace{\int e^{\|u\|^2/(2\varepsilon)} \, d\mu(u)}_{=I} \times \underbrace{\int e^{\alpha'\|x\| - \|y\|^2/(2\varepsilon)} \, d\nu(x,y)}_{=II}. 
\]
The first term $I$ is finite if, e.g., $\mu$ is compactly supported. The second term $II$ is finite for any $\alpha'>0$ if, e.g., $Y=X$, even when Condition (\ref{eq: superexponential}) fails to hold. 
\end{remark}

\subsection{Proof of Theorem \ref{thm: duality}}

\subsubsection{Preliminary lemmas}

Before the proof, we shall recall the following (standard) result. We provide its proof for completeness.

\begin{lemma}
\label{lem: minimizer}
Let $M$ be a metric space and $f: M \to (-\infty,\infty]$ be coercive in the sense that the sublevel sets of $f$ are compact. If $\inf f > -\infty$, then there exists at least one $\bar{x} \in M$ such that $f(\bar{x}) = \inf f$. 
\end{lemma}

\begin{proof}[Proof of Lemma \ref{lem: minimizer}]
The conclusion is trivial if $f \equiv \infty$, so we assume that $a := \inf f$ is finite. Let $x_n$ be a sequence such that $f(x_n) \to a$. Since $x_n \in \{ f \le a+1 \}$ for sufficiently large $n$, there exists a convergent subsequence $x_{n_k}$, $x_{n_k} \to \bar{x}$. Coercivity (in our definition) implies lower semicontinuity, so
\[
f(\bar{x}) \le \liminf_{k \to \infty} f(x_{n_k}) = a,
\]
completing the proof.
\end{proof}

We will also use the following lemma, which is a small modification to (a special case of) Theorem 3.1 in \cite{csiszar1975divergence}.

\begin{lemma}
\label{lem: csiszar}
Let $\Omega$ be a measurable space and $\calP(\Omega)$ be the set of all probability measures on $\Omega$. For given measurable functions $\phi_i: \Omega \to \R \ (i \in \{1,\dots,n\})$, let $\calQ \subset \calP(\Omega)$ be a nonempty convex set such that $\int \phi_i \, dP = 0$ for all $i \in \{1,\dots,n\}$ and for all $P \in \calQ$. For a given $R \in \calP(\Omega)$, suppose that there exists a (unique) $Q \in \calQ$ such that $\KL{Q}{R} = \inf_{P \in \calQ} \KL{P}{R} < \infty$, and that  
\begin{equation}
\calQ':=\left \{ P \in \calP(\Omega) : P \ll Q, \frac{dP}{dQ} \le 2, \int \phi_i \, dP = 0, \forall i \in \{1,\dots,n\} \right \} \subset \calQ.
\label{eq: csiszar}
\end{equation}
Then $Q$ has a density of the form
\[
\frac{dQ}{dR} = a e^{\sum_{i=1}^nb_i \phi_i}
\]
for some $a > 0$ and $b_i \in \R$.
\end{lemma} 

\begin{proof}[Proof of Lemma \ref{lem: csiszar}]
Theorem 3.1 in \cite{csiszar1975divergence} assumes that $\calQ$ (in our context) is the set of \textit{all} $P \in \calP(\Omega)$ such that $\int f_i \, dP =0$ for all $i$. In our case, we allow $\calQ$ to be smaller than that, but Condition (\ref{eq: csiszar}) ensures that $Q$ minimizes $\KL{\cdot}{R}$ over $\calQ'$, which contains $Q$ as an algebraic inner point. So, mimicking the argument in the proof of Theorem 3.1 in \cite{csiszar1975divergence} yields the claim of the lemma (recall that any finite-dimensional subspace of $L^1(Q)$ is closed in $L^1(Q)$). 
\end{proof}

The following lemma is inspired by Lemma 2.10 in \cite{nutz2021introduction}. In the proof of Theorem \ref{thm: duality} below, we will construct dual potentials from the $R$-a.e. limit of $f_n(u)+\langle g_n(u),x\rangle + h_n(x,y)$ for suitable measurable functions $(f_n,g_n,h_n)$. Because of the presence of the interaction term $\langle g_n(u),x \rangle$, special care is needed to construct dual potentials in such a way that they are separately measurable, i.e., one has to choose ``canonical'' points more carefully than standard entropic OT.  
\begin{lemma}
\label{lem: measure}
Suppose Assumption \ref{asmp: dual}~(i) holds. If a Borel subset $S \subset \R^{d_y} \times \R^{d_x+d_y}$ has full $R$-measure, then one can choose $u^* \in \R^{d_y},(x_i^*,y_i^*) \in \R^{d_x + d_y} \ (i \in \{0,1,\dots,d_x\})$ with $(u^*,x_i^*,y_i^*) \in S$ for all $i \in \{0,1,\dots,d_x\}$ for which the following hold: 
\begin{enumerate}
    \item[(i)] the $d_x+1$ vectors
    \[
    \binom{1}{x_0^*},\dots,\binom{1}{x_{d_x}^*}
    \]
    form a basis of $\R^{d_x+1}$; and
    \item[(ii)] there exist $\calU_0 \subset \R^{d_y}$ and $\calZ_0 \subset \R^{d_x+d_y}$ with $\mu(\calU_0)=\nu(\calZ_0)=1$ such that, for $S_0:=(\calU_0 \times \calZ_0) \cap S$, it holds that $(u^*,x_i^*,y_i^*) \in S_0$ for all $i \in \{0,1,\dots,d_x\}$ and
\[
(u,x,y) \in S_0 \Longrightarrow (u^*,x,y) \in S_0 \ \text{and} \ (u,x_i^*,y_i^*) \in S_0, \ \forall i \in \{0,1,\dots,d_x\}. 
\]
\end{enumerate}
\end{lemma}

\begin{proof}[Proof of Lemma \ref{lem: measure}]
For notational convenience, we write $z=(x,y)$.
For $u \in \R^{d_y}$, let $S_u : = \{ z : (u,z) \in S \}$. Define $S_z$ analogously. By Fubini's theorem, $\nu(S_u) = 1$  for $\mu$-a.e. $u$. Pick any $u^*$ for which $\nu(S_{u^*}) =1$, and set $\calZ_0 := \{ z : \mu (S_z)=1 \} \cap S_{u^*}$, which has full $\nu$-measure. By Lemma \ref{lem: nondegeneracy} below, one can choose $z_0^*,\dots,z_{d_x}^* \in \calZ_0$ (with $z_i^* = (x_i^*,y_i^*))$ for which (i) holds. Now, set $\calU_0 := \{ u : \nu(S_u)  = 1 \} \cap S_{z_0^*} \cap \cdots \cap S_{z_{d_x}^*}$, which has full $\mu$-measure. By construction, the sets $\calU_0,\calZ_0$ satisfy (ii). 
\end{proof}

\begin{lemma}
\label{lem: nondegeneracy}
Suppose  Assumption \ref{asmp: dual}~(i) holds. If a Borel subset $\calZ_0 \subset \R^{d_x + d_y}$ has  full $\nu$-measure, then
one can find $d_x+1$ points $(x_0,y_0),\dots,(x_{d_x},y_{d_x}) \in \calZ_{0}$ such that the vectors $(1,x_0^\top)^\top, \dots, (1,x_{d_x}^\top)^\top$ form a basis of $\mathbb{R}^{d_x+1}$.
\end{lemma}
\begin{proof}[Proof of Lemma \ref{lem: nondegeneracy}]
Let $\calX_{0} := \{ x  : (x,y) \in \calZ_{0} \ \text{for some} \ y \}$. Observe that $\calX_{0}$ has full measure with respect to the completion of the law of $X$. The claim of the lemma states that one can find $x_0,\dots,x_{d_x} \in \calX_{0}$ such that $(1,x_0^\top)^\top, \dots, (1,x_{d_x}^\top)^\top$ form a basis of $\mathbb{R}^{d_x+1}$.
Suppose on the contrary that there are no such $d_x+1$ vectors. Then $\{ (1,x^\top)^\top : x \in \calX_{0} \}$ is contained in a $d_x$-dimensional vector subspace of $\R^{d_x+1}$. Hence, there exists a nonzero vector $w = (w_1,w_2^\top)^\top \in \R^{d_x+1}$ such that $w_1 + w_2^\top x = 0$ for all $x \in \calX_{0}$. Observe that $w_2$ cannot be zero; otherwise $w_1 = 0$, contradicting $w \ne 0$. This, however, implies that $X$ is concentrated on an affine hyperplane, contradicting Assumption \ref{asmp: dual}~(i). 
\end{proof}

We are now ready to prove Theorem \ref{thm: duality}. 
\subsubsection{Proof of Theorem \ref{thm: duality}}
We split the proof into a few steps. 

\underline{Step 1}. We first show that there exist measurable functions $f: \R^{d_y} \to \R, g: \R^{d_y} \to \R^{d_x}$, and $h: \R^{d_x+d_y} \to \R$ such that 
\begin{equation}
\log \frac{d\pi^\varepsilon}{d\tilde{R}}(u,x,y)  = f(u) + \langle g(u),x \rangle + h(x,y).
\label{eq: density R}
\end{equation}
We shall translate the feasible set into countably many unconditional moment constraints. To this end, we first verify the following lemma. Recall the definitions of $\calQ$ in (\ref{eq: feasible set}) and $\tilde{\calP}_1$ in (\ref{eq: modified P1}). For a metric space $M$, let $BL_1(M)$ denote the set of all $1$-Lipschitz functions $M \to [-1,1]$. 
Observe that, when $M$ is $\sigma$-compact, $BL_1(M)$ is separable for the topology of uniform convergence on compacta.
\begin{lemma}
\label{lem: moment condition}
One can find countably many functions  $\{ (\frakf_i,\frakg_i,\frakh_i) : i \in \bN \}$, with $\frakf_i: \R^{d_y} \to \R, \frakg_i: \R^{d_y} \to \R^{d_x}$, and $ \frakh_i: \R^{d_x+d_y} \to \R$, such that $\frakf_i \in BL_1(\R^{d_y}), \frakh_i \in BL_1(\R^{d_x+d_y})$, and each coordinate of $\frak{g}_i$ belongs to $BL_1(\R^{d_y})$, for all $i \in \bN$, and that for a given $\pi \in \tilde{\calP}_1$,
\begin{equation}
\label{eq: moment condition}
\pi \in \calQ \Longleftrightarrow 
\begin{cases}
 \int_{\R^{d_y} \times \R^{d_x+d_y}} \frakf_i (u) \,d\pi (u,x,y) = \int \frakf_i \, d\mu \\
\int_{\R^{d_y} \times \R^{d_x+d_y}} \langle \frakg_i(u), x \rangle \,d\pi(u,x,y) = 0,\\
\int_{\R^{d_y} \times \R^{d_x+d_y}} \frakh_i(x,y) \,d\pi(u,x,y) = \int \frakh_i \, d\nu \\
\end{cases}
\forall i \in \bN.
\end{equation}
\end{lemma}

\begin{proof}[Proof of Lemma \ref{lem: moment condition}]
The first and third constraints in (\ref{eq: moment condition}) ensure that $\pi \in \Pi(\mu,\nu)$. Constructing such functions $\frakf_i$ and $\frakh_i$ is standard. For example, take $\{ \frakf_i : i \in \bN \}$ to be a countable dense subset for $BL_1(\R^{d_y})$ with respect to the topology of uniform convergence on compacta. Then the first constraint in (\ref{eq: moment condition}) holds for all $i \in \bN$ if and only if the first marginal of $\pi$ agrees with $\mu$. The construction of $\frakh_i$ is similar. 

The second constraint in (\ref{eq: moment condition}) ensures that the mean-independence constraint in (\ref{eq: entropicPrimal}) holds. Observe that
\[
\begin{split}
\int_{\R^{d_x+d_y}} x \, d\pi_u (x,y) = 0 \ \text{$\mu$-a.e. $u$} &\Longleftrightarrow 
\int_{\R^{d_y} \times \R^{d_x+d_y}} \mathbbm{1}_A(u)  x \, d\pi (u,x,y) = 0, \ \forall A \in \calB(\R^{d_y}) \\
&\Longleftrightarrow \int_{\R^{d_y} \times \R^{d_x+d_y}} f(u)  x \, d\pi (u,x,y) = 0, \ \forall f \in BL_1(\R^{d_y}).  
\end{split}
\]
The first equivalence follows by the definition of the conditional expectation. To verify the second equivalence, define a (vector-valued) signed measure on $\calB(\R^{d_y})$ by $\rho (A) := \int_{\R^{d_y} \times \R^{d_x+d_y}} \mathbbm{1}_A(u)  x \, d\pi (u,x,y)$.  The ``$\Rightarrow$''direction is trivial. For the ``$\Leftarrow$'' direction, for any closed $A \subset \R^{d_y}$, one can approximate $\mathbbm{1}_A$ by bounded Lipschitz functions to obtain $\rho(A)=0$. Applying Dynkin's $\pi$-$\lambda$ theorem, we conclude $\rho(A)=0$ for all $A \in \calB(\R^{d_y})$. 

Let $\{e_1, \dots, e_{d_x}\}$ be the standard basis in  $\mathbb{R}^{d_x}$, and take $\frakg_{i,j} = \frakf_i \cdot e_j$. Observe that, by the dominated convergence theorem, 
\[
\begin{split}
&\int_{\R^{d_y} \times \R^{d_x+d_y}} f(u)  x \, d\pi (u,x,y) = 0, \ \forall f \in BL_1(\R^{d_y})  \\
&\Longleftrightarrow \int_{\R^{d_y} \times \R^{d_x+d_y}} \langle \frakg_{i,j}(u), x \rangle \,d\pi(u,x,y) = 0, \ \forall i \in \bN, \forall j \in \{ 1,\dots,d_x \}. 
\end{split}
\]
This completes the proof. 
\end{proof}

For each $n$, define the set 
\[
\calQ_n := \left\{\pi \in \tilde{\calP}_1 : \text{(\ref{eq: moment condition}) holds all $i \le n$} \right\}.
\]
Each $\calQ_n$ is nonempty, convex, closed in $\tilde{\mathsf{W}}_1$ (cf. Theorem 6.9 in \cite{villani2009optimal}), and nonincreasing in $n$ with $\bigcap_{n=1}^\infty \calQ_n = \calQ$.  
Now, since $Q \mapsto \KL{Q}{\tilde{R}}$ is strictly convex (on its domain) and coercive in $\tilde{\mathsf{W}}_1$ (the coercivity is guaranteed by Assumption \ref{asmp: dual}~(ii)), by Lemma \ref{lem: minimizer}, 
there exists a (unique) $\pi_n \in \calQ_n$ such that
\[
\KL{\pi_n}{\tilde{R}} = \inf_{\pi \in \calQ_n} \KL{\pi}{\tilde{R}}. 
\]
As Condition (\ref{eq: csiszar}) above holds for $\calQ = \calQ_n, R=\tilde{R}$, $Q=\pi_n$, and $\phi_i$'s suitably chosen,  we can apply Lemma \ref{lem: csiszar} to conclude that there exist $(f_n,g_n,h_n) \in L^\infty(\mu) \times L^\infty(\mu;\R^{d_x}) \times L^\infty(\nu)$ such that
\[
\log \frac{d\pi_n}{d\tilde{R}} (u,x,y)  = f_n(u) + \langle g_n(u),x \rangle + h_n(x,y).
\]
These are constructed as follows: $f_n$ is a linear combination of $\{ 1,\frakf_1,\dots,\frakf_n \}$, and $g_n,h_n$ are linear combinations of $\{ \frakg_i : i \le n\}$ and $\{ \frakh_i : i \le n \}$, respectively. 

Arguing exactly as in the proof of \cite[Proposition 1.17]{nutz2021introduction}, we have
\begin{align}
&\pi_n \to \pi^\varepsilon \ \text{in total variation}, \ \text{i.e.}, \ \frac{d\pi_n}{d\tilde{R}} \to \frac{d\pi^\varepsilon}{d\tilde{R}} \ \text{in $L^1(\tilde{R})$}, \label{eq: tv conv} \\
&\log \frac{d\pi_n}{d\tilde{R}} \to \log \frac{d\pi^\varepsilon}{d\tilde{R}} \ \text{in $L^1(\pi^\varepsilon)$}.  \label{eq: L1}
\end{align}
In \cite[Proposition 1.17]{nutz2021introduction}, $\calQ_n$ is assumed to be closed in total variation, which does not hold unless $X$ is bounded in our case, but the said assumption is used to guarantee the existence of entropic projection $\pi_n$ of $\tilde{R}$ onto $\calQ_n$, which was established via a different route in our case. Given the existence of $\pi_n$, the argument in the proof of \cite[Proposition 1.17]{nutz2021introduction} goes through verbatim. 

Now, by (\ref{eq: tv conv}), after passing to a subsequence if necessary and using $\tilde{R} \sim R$, we conclude 
\begin{equation}
F(u,x,y):=\log \frac{d\pi^\varepsilon}{d\tilde{R}}(u,x,y) = \lim_{n \to \infty} \big( f_n(u) + \langle g_n(u),x \rangle + h_n(x,y) \big) \in [-\infty,\infty)
\label{eq: limit}
\end{equation}
for $R$-a.e. $(u,x,y)$. Recall that the feasible set $\calQ$ is convex and that $\pi^\varepsilon$ minimizes $\KL{\pi}{\tilde{R}}$ over all $\pi \in \calQ$. Since $\KL{R}{\tilde{R}} < \infty$,  Corollary 1.13 of \cite{nutz2021introduction} yields that $F \in L^1(R)$, which implies that $F$ is finite $R$-a.e.
Let $S \subset \R^{d_x+d_y}$ be the set of full $R$-measure on which the limit (in $\R$) on the right-hand side of (\ref{eq: limit}) exists. 

We invoke Lemma \ref{lem: measure} above to find sets $\calU_0 \subset \R^{d_y}, \calZ_0 \subset \R^{d_x+d_y}$ with $\mu(\calU_0)=\nu(\calZ_0)=1$ and points $u^* \in \R^{d_y}, (x_i^*,y_i^*) \in \R^{d_x+d_y} \ (i \in \{ 0,1,\dots,d_x \})$ for which (i) and (ii) in Lemma \ref{lem: measure} hold. Recall $S_0 = (\calU_0 \times \calZ_0) \cap S$. 
If we define 
\[
\begin{split}
&\tilde{f}_n(u) = f_n(u) - f_n(u^*), \\
&\tilde{g}_n(u) = g_n(u) - g_n(u^*), \\
&\tilde{h}_n(x,y) = h_n(x,y) + f_n(u^*) + \langle g_n(u^*),x \rangle,
\end{split}
\]
then $\tilde{f}_n(u^*)=0, \tilde{g}_n(u^*)=0$, and 
\[
F_n(u,x,y):=f_n(u)+\langle g_n(u),x \rangle + h_n(x,y) \\
=\tilde{f}_n(u)+\langle \tilde{g}_n(u),x \rangle + \tilde{h}_n(x,y).
\]
For notational convenience, relabel $(\tilde{f}_n,\tilde{g}_n,\tilde{h}_n)$ as $(f_n,g_n,h_n)$. 
Set $\calU_{00} := \{ u \in \calU_0 : (u,x,y) \in S_0 \ \text{for some} \ (x,y) \in \calZ_0 \}$ and $\calZ_{00} := \{ (x,y) \in \calZ_0 : (u,x,y) \in S_0 \ \text{for some} \ u \in \calU_0 \}$. These sets need not be Borel measurable, but are analytic and hence $\mu$- and $\nu$-completion measurable, respectively (cf. Chapter 13 in \cite{Dudley_2002}). Observe that $\bar{\mu}(\calU_{00}) = \bar{\nu}(\calZ_{00})=1$, with $\bar{\mu}, \bar{\nu}$ denoting the completions of $\mu, \nu$, respectively. 

Pick any $(x,y) \in \calZ_{00}$.
Since $(u^*, x,y) \in S_0$, we have
\[
F(u^*, x,y) = \lim_{n\to\infty} F_n(u^*, x,y) = \lim_{n \to \infty} h_n(x,y).
\]
So let us define
\[
h(x,y) := 
\begin{cases}
    F(u^*, x,y) & \text{if $(x,y) \in \calZ_{00}$,} \\
    0 & \text{otherwise},
\end{cases}
\]
and $G(u,x,y) := F(u,x,y) - h(x,y)$. On $S_0$, one has
\[
G(u, x,y) = \lim_{n\to\infty} (F_n(u, x,y) - h_n(x,y)) 
= \lim_{n\to\infty} \big ( f_n(u) + \langle g_n(u),x \rangle \big ).
\]
Let $H_n(u) := (f_n(u), g_n(u)^{\top})^{\top}$ be a vector-valued mapping $H_n: \R^{d_y} \to \mathbb{R}^{d_x+1}$, and let $v(x,y) := (1,x^\top)^\top$ be another vector-valued mapping $v: \R^{d_x+d_y} \to \mathbb{R}^{d_x+1}$. We have
\[
G(u,x,y) = \lim_{n \to \infty}\langle H_n(u),v(x,y) \rangle.
\]
Pick any $u \in \calU_{00}$. Observe that $(u,x_i^*,y_i^*) \in S_0$ for all $i \in \{0,1,\dots,d_x \}$. 
Let $V$ be the invertible $(d_x+1) \times (d_x+1)$ matrix whose rows are $v(x_i^*,y_i^*)^\top$, and define
\[
\varphi_n(u) := \begin{pmatrix} \langle H_n(u), v(x_0^*,y_0^*) \rangle \\ \vdots \\ \langle H_n(u), v(x_{d_x}^*,y_{d_x}^*) \rangle \end{pmatrix} = V  H_n(u).
\]
By construction, the limit of  $\varphi_n(u)$ exists for $u \in \calU_{00}$, so define
\[
\varphi(u) := \lim_{n\to\infty} \varphi_n(u).
\]
Finally, define
\[
\big(f(u), g(u)^\top\big)^\top:= \lim_{n\to\infty} H_n(u) = \lim_{n\to\infty} V^{-1} \varphi_n(u) = V^{-1} \varphi(u), \ u \in \calU_{00},
\]
and set $(f(u), g(u)^\top)^\top = 0$ outside $\calU_{00}$. 
The functions $f$ and $g$, as constructed, are $\mu$-completion measurable. They need not be Borel measurable, but one can find Borel measurable $\tilde{f}: \R^{d_y} \to \R$ and $\tilde{g}: \R^{d_y} \to \R^{d_x}$ such that $f$ and $g$ differ from $\tilde{f}$ and $\tilde{g}$, respectively, only on $\mu$-null sets. Likewise, one can find Borel measurable $\tilde{h}: \R^{d_x+d_y} \to \R$ that differs from $h$ only on a $\nu$-null set. For notational convenience, relabel $(\tilde{f},\tilde{g},\tilde{h})$ as $(f,g,h)$. 
We have shown that
\[
\log \frac{d\pi^\varepsilon}{d\tilde{R}}(u,x,y)  = F(u,x,y) = f(u) + \langle g(u),x \rangle + h(x,y)
\]
for $R$-a.e. $(u,x,y)$.

\underline{Step 2}. Next, we shall show that $(f,g,h) \in L^1(\mu) \times L^1(\mu;\R^{d_x}) \times L^1(\nu)$. Recall that $F \in L^1(R)$. By Fubini's theorem, the function $\tilde{F}(x, y) := \int_{\R^{d_y}} F(u, x, y)\, d\mu(u)$ is in $L^1(\nu)$. Since
\[
\begin{split}
\tilde{F}(x, y) &= \int_{\R^{d_y}} \big(f(u) + \langle g(u),x \rangle \big) \,d\mu(u) + h(x, y) 
\end{split}
\]
we see that $K(x):=\int_{\R^{d_y}} \big(f(u) + \langle g(u),x \rangle \big) \,d\mu(u)$ is finite for all $(x,y)$ on a set of full $\nu$-measure. 

Now, by Lemma \ref{lem: nondegeneracy}, one can find $d_x+1$ points $x_0, x_1, \dots, x_{d_x}$ such that $K(x_i)$ are finite for all $i \in \{0,\dots,d_x\}$, and that $v_i := x_i - x_0$ ($i \in \{ 1,\dots,d_x\}$) form a basis of $\R^{d_x}$. 
Since
$K(x_i) - K(x_0) = \int_{\R^{d_y}} \langle g(u),v_i \rangle \, d\mu(u)$,
we see that $\langle g(\cdot),v_i \rangle \in L^1(\mu)$ for all $i \in \{1,\dots,d_x\}$. 
This implies that $g \in L^1(\mu;\R^{d_x})$ and  $f \in L^1(\mu)$. Finally, since 
$
|h(x,y)| \le |\tilde{F}(x,y)| + \| f \|_{L^1(\mu)} + \| g \|_{L^1(\mu)} \| x \|$, 
$\tilde{F} \in L^1(\nu)$, and $\E[\|X\|] < \infty$, 
we conclude that $h \in L^1(\nu)$. In addition, since $|\langle g(u),x\rangle| \le |F(u,x,y)| + |f(u)| + |h(x,y)|$ and $F \in L^1(\pi^\varepsilon)$ by (\ref{eq: L1}), we have that $\langle g(u),x\rangle$ is integrable under $\pi^\varepsilon$. 

We have shown that there exist $(f,g,h) \in L^1(\mu) \times L^1(\mu;\R^{d_x}) \times L^1(\nu)$ such that (\ref{eq: density R}) holds. 
Since
\[
\begin{split}
\log \frac{d\pi^\varepsilon}{dR}(u,x,y) &= \log \frac{d\pi^\varepsilon}{d\tilde{R}}(u,x,y) + \log \frac{d\tilde{R}}{dR}(u,x,y) \\
&= f(u) + \langle g(u),x \rangle + h(x,y) - \log \alpha - \frac{1}{\varepsilon}c(u,y),
\end{split}
\]
the expression (\ref{eq: density}) holds with $f^\varepsilon = \varepsilon(f-\log \alpha), g^\varepsilon = \varepsilon g$, and $h^\varepsilon = \varepsilon h$. 

\underline{Step 3}. Finally, we establish the conclusion of the theorem. Observe that
\[
\begin{split}
    \varepsilon \KL{\pi^{\varepsilon}}{R} &=  \varepsilon \int \log\left(\frac{d\pi^{\varepsilon}}{dR}\right) \,d\pi^{\varepsilon} \\
    &= \int  \big ( f^\varepsilon (u) + \langle g^\varepsilon (u),x \rangle + h^\varepsilon(x,y) -c(u,y) \big)\, d\pi^{\varepsilon}(u,x,y) \\
    &=\int f^\varepsilon \, d\mu + \int h^\varepsilon \, d\nu - \int c \, d\pi^\varepsilon. 
\end{split}
\]
For the third equality, we used the fact that, for $(U,\tilde{X},\tilde{Y}) \sim \pi^\varepsilon$,  we have $\E[|\langle g^\varepsilon (U),\tilde{X}\rangle|] < \infty$ as we have verified before and
\[
\int \langle g^\varepsilon (u),x \rangle \, d\pi^\varepsilon(u,x,y)=\E\big[\langle g^\varepsilon (U),\tilde{X}\rangle\big] =\E\big[\langle g^\varepsilon (U), \E[\tilde{X} \mid U] \rangle\big] = 0.
\]
Observe that $\iota^\varepsilon(f^\varepsilon,g^\varepsilon,h^\varepsilon)=0$ as $\pi^\varepsilon$ is a probability measure. 
Conclude that
\[
\begin{split}
\mathsf{T}^\varepsilon (\mu,\nu) &= \int c \,d\pi^\varepsilon + \varepsilon\KL{\pi^\varepsilon}{R} \\
&= \int f^\varepsilon \, d\mu + \int h^\varepsilon \, d\nu  \\
&= D^\varepsilon (f^\varepsilon,g^\varepsilon,h^\varepsilon) \le \mathsf{D}^\varepsilon(\mu,\nu). 
\end{split}
\]
Combining the weak duality from Proposition \ref{prop: weak duality}, the strong duality holds, and $(f^\varepsilon,g^\varepsilon,h^\varepsilon)$ solve the dual problem. This completes the proof. 
\qed

\subsection{Proof of Proposition \ref{prop: uniqueness}}
Observe that the dual objective can be written as
\[
\begin{split}
D^\varepsilon (f,g,h) &= \int \big(f(u)+\langle g(u),x\rangle + h(x,y) \big) \, dR(u,x,y) \\
&\quad - \varepsilon \int e^{\frac{1}{\varepsilon} (f(u)+\langle g(u),x\rangle + h(x,y) -c(u,y))} \, dR(u,x,y) + \varepsilon
\end{split}
\]
for $(f,g,h) \in L^1(\mu) \times L^1(\mu;\R^{d_x}) \times L^1(\nu)$. 
If $(f,g,h), (\tilde f,\tilde g,\tilde h)$ both maximize the dual objective, then
by strict convexity of the exponential function, we have
\[
f(u) + \langle g(u), x\rangle + h(x,y) = \tilde{f}(u) + \langle \tilde{g}(u),x\rangle + \tilde{h}(x,y) \quad \text{$R$-a.e. $(u,x,y)$}.
\]
Define
\[
\frakf(u) = \tilde{f}(u) - f(u), \quad \frakg(u) = \tilde{g}(u) - g(u), \quad \frakh(x,y) = \tilde{h}(x,y) - h(x,y).
\]
We have
\[
\frakh(x,y) = - \left( \frakf (u) + \langle \frakg(u), x \rangle \right) \quad \text{$R$-a.e. $(u,x,y)$}.
\]
Since the left-hand side does not depend on $u$, there exists $u_0 \in \R^{d_y}$ such that
\[
\frakh(x,y)=-a-\langle v,x \rangle \quad \text{$\nu$-a.e. $(x,y)$,}
 \]
where $a=\frakf(u_0)$ and $v = \frakg(u_0)$. This implies that
 \[
(\frakf(u) - a) + \langle \frakg(u)- v,x \rangle = 0 \quad \text{$R$-a.e. $(u,x,y)$}.
\]
Under Assumption \ref{asmp: dual}~(i) (cf. Lemma \ref{lem: nondegeneracy}),  one has $\frakf - a = 0$ and $\frakg- v = 0$ $\mu$-a.e.
\qed

\subsection{Proof of Proposition \ref{prop: converse}}
    Observe that 
    \[
    \int \left(\log \frac{d\pi}{dR}\right)^- \, d\pi  = \int \left(\log \frac{d\pi}{dR}\right)^- \frac{d\pi}{dR} \, dR <  \infty
    \]
    because $a\log a \ge -e^{-1}$ for $a \ge 0$. 
    Since
    \[
        (\langle g(u),x\rangle)^- \le \varepsilon \left(\log \frac{d\pi}{dR}\right)^- + \left(f(u) + h(x,y) - c(u,y)\right)^+,
    \]
    it follows that $(\langle g(u),x\rangle)^-  \in L^1(\pi)$. Arguing as in the proof of Proposition \ref{prop: weak duality}, we see that $\langle g(u),x\rangle \in L^1(\pi)$,  which ensures $\int_{\R^{d_y} \times \R^{d_x+d_y}} \langle g(u),x\rangle d\pi(u,x,y) = 0$ as $\pi \in \calQ$. The rest follows from weak duality. 
    \qed
    

\subsection{Proofs for Section \ref{sec: regularity}}
Let $\kappa$ denote the marginal distribution of $X$, so that $d\nu = d\nu_x d\kappa$ (recall that $\nu_x$ denotes the conditional distribution of $Y$ given $X$).

\begin{proof}[Proof of Proposition \ref{prop: continuity}]
 Let $y_1, y_2$ be such that $h^\varepsilon(x, y_1)  \ge h^\varepsilon(x, y_2)$. Let $H(u, x, y) := f^\varepsilon(u) + \langle g^\varepsilon(u),x\rangle - c(u,y)$, so that
$h(x,y_1) = - \varepsilon \log \int e^{H(u,x,y_1)/\varepsilon} \,d\mu(u) > -\infty$. Then
\[
\begin{split}
h^\varepsilon(x, y_2) &= - \varepsilon \log \int e^{H(u,x,y_1)/\varepsilon} \cdot  e^{(c(u,y_1)-c(u,y_2))/\varepsilon} \,d\mu(u) \\
 &\ge -\sup_{u \in \calU}|c(u,y_1)-c(u,y_2)| + h^\varepsilon(x,y_1)\\
&> -\infty.
\end{split}
\]
Now, we observe
\[
\begin{split}
0 \le h^\varepsilon(x, y_1) - h^\varepsilon(x, y_2) & =  \varepsilon \log \int e^{H(u,x,y_2)/\varepsilon} \,d\mu(u) - \varepsilon \log \int e^{H(u,x,y_1)/\varepsilon}  \,d\mu(u) \\
 &= \varepsilon \log \int e^{H (u,x,y_1)/\varepsilon} \cdot  e^{(c(u,y_1)-c(u,y_2))/\varepsilon}  \,d\mu(u) \\
 &\qquad - \varepsilon \log \int e^{H(u,x,y_1)/\varepsilon}  \,d\mu(u) \\
&\le \sup_{u \in \calU}(c(u,y_1)-c(u,y_2)), 
\end{split}
\] 
completing the proof.
\end{proof}

\begin{proof}[Proof of Proposition \ref{prop: convexity}]
For notational convenience, we set $\varepsilon = 1$ and omit the dependence on $\varepsilon$. The general case follows analogously. 
Let $w(u) := e^{f(u)-c(u,y)}$ (recall that $y$ is fixed)  and  $I(x) := \int e^{\langle g(u), x \rangle} w(u) \,d\mu(u)$.  Since $h(x, y) = -\varepsilon \log I(x)$, to show that $h(x,y)$ is concave in $x$, it suffices to verify that $\log I$ is convex. This follows from the fact that $I(x)$ has the form of a Laplace transform (cf. \cite[Theorem 7.1]{barndorff2014information}), but we include a self-contained proof for completeness. Pick any $\lambda \in (0,1)$. By H\"older's inequality, 
\[
\begin{split}
I(\lambda x_1 + (1-\lambda)x_2) &= \int e^{\langle g(u), \lambda x_1 + (1-\lambda)x_2 \rangle} w(u) \,d\mu(u) \\
 &= \int \left( e^{\langle g(u), x_1 \rangle} w(u) \right)^\lambda \left( e^{\langle g(u), x_2 \rangle} w(u) \right)^{1-\lambda} \,d\mu(u) \\
&\le \left( \int e^{\langle g(u), x_1 \rangle} w(u) \,d\mu(u) \right)^\lambda \left( \int e^{\langle g(u), x_2 \rangle} w(u) \,d\mu(u) \right)^{1-\lambda}\\
&= I(x_1)^\lambda I(x_2)^{1-\lambda}.
\end{split}
\]
This implies convexity of $\log I$. 

To prove the second claim, define $D := \{x \in \R^{d_x}: h(x, y) > -\infty\}$, which is convex. Let $S$ be the convex hull of $\calX$. Since $\kappa(D) = 1$, we see that $S \subset \bar{D}$, where $\bar{D}$ denotes the closure of $D$. 
 We shall show that $S$ has nonempty interior. Indeed, $0$ is in the interior of $S$, since otherwise there exists a nonzero vector $v \in \mathbb{R}^{d_x}$ such that $\langle v,x \rangle \ge 0$ for all $x \in S$ by the separating hyperplane theorem (cf. \cite[Exercise 1.9]{brezis2011functional}), which implies $\langle v,X \rangle=0$ a.s. because $\E[X]=0$, contradicting Assumption \ref{asmp: dual}~(i).  

Next, we shall show that $D$  has nonempty interior. Otherwise,  \cite[Theorem 6.2.6]{Dudley_2002} yields that $\bar{D}$ is contained in some hyperplane, which contradicts the fact that $S$ has nonempty interior. Now, since $D$ has nonempty interior,  the interior of $D$ agrees with the interior of $\bar{D}$ (see, e.g., \cite[Exercise 1.7]{brezis2011functional}). As such, the interior of $S$ is contained in the interior of $D$. The conclusion follows from the fact that a convex function is locally Lipschitz on the interior of its domain \cite[Theorem 3.1.11]{nesterov2018lectures}. 
\end{proof}

\begin{proof}[Proof of Proposition \ref{prop: exponential}]
Again, we set $\varepsilon = 1$ and omit the dependence on $\varepsilon$.
 Let $w_u(x,y) := e^{h(x,y)-c(u,y)}$, so equation (\ref{eq: implicit}) reduces to
 \begin{equation}
 \int x e^{\langle \theta,x\rangle} w_u(x,y) \, d\nu(x,y) = 0. \label{eq: implicit 2}
 \end{equation}
Define a Borel measure $\tilde{\kappa}_u$ on $\R^{d_x}$ by
\begin{equation}
\frac{d\tilde{\kappa}_u(x)}{d\kappa(x)} := \int w_u(x,y) \, d\nu_x  (y) =: \tilde{w}_u(x).
\label{eq: kappa}
\end{equation}
Since $w_u(x,y)$ is strictly positive for all $(u,y)$ and for $\kappa$-a.e. $x$, $\tilde{\kappa}_u$ and $\kappa$ are mutually absolutely continuous. 
We shall verify that $\tilde{\kappa}_u$ is a finite measure for all $u \in \R^{d_y}$. Observe that, for any $u_0, u \in \R^{d_y}$, 
\[
\tilde{w}_u(x)\le e^{\sup_{y \in \calY}|c(u_0,y)-c(u,y)|} \tilde{w}_{u_0}(x).
\]
As such, it suffices to verify that $\tilde{\kappa}_{u_0}$ is a finite measure for some $u_0$. 
Recall that ($\pi = \pi^{\varepsilon}$ with $\varepsilon=1$)
\[
d\pi (u,x,y)= w_u(x,y) e^{f(u)+\langle g(u),x \rangle} \, d\mu (u)d\nu_x(y) d\kappa(x).
\]
Since the first marginal of $\pi$ agrees with $\mu$ by construction, one can choose $u_0 \in \calU$ such that $f(u_0)$ and $g(u_0)$ are both finite and the conditional distribution $\pi_{u_0}$ agrees with 
\[
d\pi_{u_0}(x,y) = w_{u_0}(x,y) e^{f(u_0)+\langle g(u_0),x \rangle} d\nu_x(y) d\kappa(x),
\]
that is,
\[
w_{u_0}(x,y) d\nu_x(y) d\kappa(x) =  e^{-f(u_0)-\langle g(u_0),x \rangle} d\pi_{u_0}(x,y).
\]
Since $\calX$ is compact, 
\[
\begin{split}
\tilde{\kappa}_{u_0} (\R^{d_x}) &= \int_{\R^{d_x+d_y}} e^{-f(u_0)-\langle g(u_0),x \rangle} d\pi_{u_0}(x,y) \\
&\le e^{-f(u_0) + \| g(u_0) \| \sup_{x \in \calX}\|x\|} < \infty.
\end{split}
\]

We have verified that $\tilde{\kappa}_u$ is a finite measure for all $u \in \R^{d_y}$. From now on, we pick and fix any $u \in \R^{d_y}$ and omit the dependence on $u$. 
Define
\[
p_{\theta}(x) := e^{\langle \theta,x \rangle - A(\theta)} \quad \text{with} \quad A(\theta) := \log \int  e^{\langle \theta, x \rangle} 
\,d\tilde{\kappa}(x).
\]
Since $\tilde{\kappa}$ is supported in $\calX$, $A(\theta)$ is finite for all $\theta \in \R^{d_x}$, and $\{ p_\theta : \theta \in \R^{d_x} \}$ constitutes an exponential family of densities with base measure $\tilde{\kappa}$. Furthermore, Assumption \ref{asmp: dual}(i) guarantees that there is no nonzero vector $v \in \R^{d_x}$ such that $\langle v,x \rangle$ is constant $\kappa$-a.e. (or equivalently  $\tilde{\kappa}$-a.e.), so the exponential family $\{ p_\theta : \theta \in \R^{d_x} \}$ is \textit{minimal} \cite[p.40]{wainwright2008graphical}. To establish the conclusion of the proposition, we will use theory of exponential families (cf. Chapter 3 in \cite{wainwright2008graphical}).

By \cite[Proposition 3.1]{wainwright2008graphical}, $A(\theta)$ has derivatives of all orders on $\R^{d_x}$. Observe that 
\[
\text{(\ref{eq: implicit 2})} \Longleftrightarrow \nabla A(\theta) = 0.
\]
By Proposition 3.2 and Theorem 3.3 of \cite{wainwright2008graphical}, $\nabla A$ is a bijection between $\R^{d_x}$ and the interior of the \textit{mean parameter space} $\calM:= \{ \int x \, d\rho(x) : \rho \in \calP_1(\R^{d_x}),  \rho \ll \kappa \}$ (recall that $\tilde{\kappa}$ and $\kappa$ are mutually absolutely continuous). 
Lemma \ref{lem: mean parameter} below shows that the interior of $\calM$ contains $0$.
Hence, there exists a unique $\theta$ with $\nabla A(\theta) = 0$. 
\end{proof}

It remains to prove Lemma \ref{lem: mean parameter} that was used in the proof of Proposition \ref{prop: exponential}. 

\begin{lemma}
\label{lem: mean parameter}
Recall the mean parameter space $\calM = \{ \int x d\rho(x) : \rho \in \calP_1(\R^{d_x}), \rho \ll \kappa \}$. Under Assumption \ref{asmp: dual}~(i), the interior of $\calM$ contains $0$. 
\end{lemma}

\begin{proof}
Observe that $\calM$ is convex and $0 \in \calM$. Suppose on the contrary that $0$ is not in the interior of $\calM$. Then, by the separating hyperplane theorem, there exists a nonzero vector $v \in \R^{d_x}$ such that $\langle v,x \rangle \ge 0$ for all $x \in \calM$. For any $A \in \calB(\R^{d_x})$ with $\kappa (A) > 0$, we see that $\frac{\E[X\mathbbm{1}_A(X)]}{\kappa(A)} \in \calM$ (choose $\rho$ as $d\rho = \frac{1}{\kappa (A)} \mathbbm{1}_{A} \, d\kappa$). Hence,
\[
\E[ \langle v,X \rangle \mathbbm{1}_{A}(X)] \ge 0, \ \forall A \in \calB(\R^{d_x}).
\]
This implies $\langle v, X \rangle \ge 0$ a.s. and hence $\langle v, X \rangle = 0$ a.s. because $\E[X]=0$. But this contradicts Assumption \ref{asmp: dual}~(i). 
\end{proof}

\begin{proof}[Proof of Theorem \ref{thm: analytic}]
Again, we set $\varepsilon = 1$ and omit the dependence on $\varepsilon$.

(i) Let $A(\theta, u)$ be defined by
\[
\begin{split}
e^{A(\theta, u)} &= \int e^{\langle \theta, x \rangle + h(x,y) - c(u,y)} \,d\nu(x,y) \\
& = e^{-\|u\|^2/2} \underbrace{\int e^{\langle \theta,x\rangle + \langle u,y \rangle + h(x,y)-\|y\|^2/2} \,d\nu(x,y)}_{=:\mathsf{Z}(\theta, u)},
\end{split}
\]
so that $A(\theta,u) = -\|u\|^2/2 + \log \mathsf {Z}(\theta,u)$. Observe that $\log \mathsf{Z}(\theta,u)$ is the log-partition function corresponding to the exponential family $e^{\langle \theta,x\rangle + \langle u,y \rangle - \log \mathsf{Z}(\theta,u)}$ with base measure $e^{h(x,y)-\|y\|^2/2} d\nu(x,y)$. The base measure is finite by a similar argument to the previous proof, and $\mathsf{Z}(\theta,u)$ is finite for all $(\theta,u) \in \R^{d_x+d_y}$ because $\nu$ is compactly supported. 
Now, by \cite[Theorem 7.2]{barndorff2014information}, $\log Z(\theta,u)$ is real analytic on $\R^{d_x+d_y}$, and so is $A(\theta, u)$.

Observe that for $u \in \R^{d_y}$, $g(u)$ is the unique vector in $\R^{d_x}$ satisfying $\nabla_{\theta} A(g(u),u)=0$. Furthermore, one can write
\[
A(\theta, u) = \log \int_{\mathbb{R}^{d_x}} e^{\langle \theta,x \rangle} \,d\tilde{\kappa}_u(x),
\]
where $\tilde{\kappa}_u$ is defined by (\ref{eq: kappa}). 
By \cite[Proposition 3.1(a)]{wainwright2008graphical}, $\nabla_\theta^2 A(\theta, u)$ is symmetric positive semidefinite. Minimality of the exponential family $\{ e^{\langle \theta,x \rangle - A(\theta,u)} : \theta \in \R^{d_x} \}$ (with base measure $\tilde{\kappa}_u$), as verified in the proof of the previous proposition, implies that $\nabla_\theta^2 A(\theta, u)$ is indeed positive definite; see the proof of \cite[Proposition 3.1(b)]{wainwright2008graphical}.  
 Now, since $A$ is real analytic, applying the real analytic implicit function theorem \cite[Theorem 2.3.5]{krantz2002primer},  each coordinate of $g$ is real analytic. 

Furthermore, by definition, $f(u) = - A(g(u), u)$ (which in particular implies that $f$ is everywhere finite). Since $A$ and $g$ are both real analytic, their composition is real analytic. 

(ii) Observe that 
\[
h(x,y) = -\|y\|^2/2 - \log \int e^{\langle g(u),x \rangle +\langle u,y \rangle} e^{f(u)-\|u\|^2/2} \, d\mu(u).  
\]
The right-hand side is finite for all $(x,y) \in \R^{d_x+d_y}$ since $f$ and $g$ are bounded on $\calU$. The conclusion follows from \cite[Theorem 7.2]{barndorff2014information}.
\end{proof}

\section{Proofs for Section \ref{sec: gaussian}}
\label{sec: proof gaussian}

\subsection{Proof of Theorem \ref{thm: gaussian solution}~(i)}

We first verify that the optimal coupling must be Gaussian. 
\begin{lemma}
\label{lem: Gaussian optimality}
If $\mu$ and $\nu$ are both nondegenerate Gaussian, then the optimal coupling $\pi^\varepsilon$  for (\ref{eq: entropicPrimal}) is nondegenerate Gaussian.
\end{lemma}
\begin{proof}[Proof of Lemma \ref{lem: Gaussian optimality}]
For a Borel probability measure $\rho$ on a finite dimensional Euclidean space with $d\rho (x) = f(x) \, dx$, let $\Ent (\rho)$ denote the differential entropy, 
\[
\Ent (\rho) :=-\int f(x) \log f(x) \, dx.
\]
Any coupling $\pi \in \Pi(\mu,\nu)$ with $\KL{\pi}{\mu \otimes \nu} < \infty$
is absolutely continuous with respect to Lebesgue measure, and the KL divergence decomposes as 
\[
\KL{\pi}{\mu\otimes\nu} = \Ent(\mu) + \Ent(\nu) - \Ent(\pi).
\]
 Since $\Ent(\mu)$ and $\Ent(\nu)$ depend only on the marginals,  the entropic VQR problem (\ref{eq: entropicPrimal}) is equivalent to
\begin{equation}
\label{eq: equivalent formulation}
\inf_{\pi \in \Pi(\mu,\nu)} \left \{ \E[\|U-\tilde{Y}\|^2/2] - \varepsilon \Ent (\pi) : (U,\tilde{X},\tilde{Y}) \sim \pi, \ \E[\tilde{X} \mid U ] = 0 \ \text{a.s.} \right \}.
\end{equation}
Let $\pi_G$ be the Gaussian measure that shares the same mean vector and covariance matrix as $\pi^\varepsilon$. Since the marginals are Gaussian, $\pi_G$ is a coupling for $(\mu,\nu)$. We shall show that $\pi_G = \pi^\varepsilon$. 
To simplify notation, below, $\E_{\pi}$ means that the expectation is taken with respect to $(U,\tilde{X},\tilde{Y}) \sim \pi$. 

First, $\E_{\pi^\varepsilon}[\tilde{X} \mid U] = 0$ implies $\E_{\pi^\varepsilon}\big[\tilde{X} U^\top\big] = 0$, and by construction, $\E_{\pi_G}\big [\tilde{X}U^\top \big] = 0$. 
For a multivariate Gaussian distribution, uncorrelatedness implies independence, so $\tilde{X} \independent U$ under $\pi_G$. This implies $\mathbb{E}_{\pi_G}[\tilde{X} \mid U] = 0$ a.s. and that $\pi_G$ is feasible for (\ref{eq: equivalent formulation}). 

Next, since $\pi^\varepsilon$ and $\pi_G$ share the same mean vector and covariance matrix, we have $\E_{\pi_G}[\|U-\tilde Y\|^2] = \E_{\pi^\varepsilon}[\|U-\tilde Y\|^2]$.
In addition, identifying a measure and its Lebesgue density, we have 
\[
\begin{split}
0 \le \KL{\pi^\varepsilon}{\pi_G} &= \int \pi^\varepsilon (w) \log \frac{\pi^\varepsilon (w)}{\pi_G(w)} \, dw \\
&=-\Ent (\pi^\varepsilon) + \E_{\pi^\varepsilon}\left [-\log \pi_G(U,\tilde{X},\tilde{Y}) \right] \\
&= -\Ent (\pi^\varepsilon) + \underbrace{\E_{\pi_G}\left [-\log \pi_G(U,\tilde{X},\tilde{Y}) \right]}_{=\Ent(\pi_G)},
\end{split}
\]
that is, $\text{Ent}(\pi_G) \geq \text{Ent}(\pi^\varepsilon)$.
Since $\pi^\varepsilon$ is the unique minimizer of (\ref{eq: equivalent formulation}), we conclude $\pi^\varepsilon = \pi_G$.
    \end{proof}

\begin{proof}[Proof of Theorem \ref{thm: gaussian solution}~(i)]
Lemma \ref{lem: Gaussian optimality} implies that $\pi^\varepsilon = N(m, \Gamma_\varepsilon)$ with 
\[
m = \begin{pmatrix}0 \\ 0 \\ m_Y \end{pmatrix} \quad \text{and} \quad \Gamma_\varepsilon = \begin{pmatrix} I_{d_y} & O & \Lambda_{\varepsilon} \\ O & \Sigma_{XX} & \Sigma_{XY} \\ \Lambda_{\varepsilon}^\top & \Sigma_{YX} & \Sigma_{YY} 
\end{pmatrix}
.
\]
 where $\Lambda_{\varepsilon}$ is the only unknown. The matrix $\Gamma_\varepsilon$ is nonsingular because $\pi^\varepsilon$ is nondegenerate. Partition the precision matrix $\Theta := \Gamma_\varepsilon^{-1}$ as 
 \[
\Theta =
 \begin{pmatrix}
\Theta_{UU} & \Theta_{UX} & \Theta_{UY} \\
\Theta_{XU} & \Theta_{XX} & \Theta_{XY} \\
\Theta_{YU} & \Theta_{YX} & \Theta_{YY}
 \end{pmatrix}
=
\begin{pmatrix}
\Theta_{UU} & \Theta_{UZ} \\
\Theta_{ZU} & \Theta_{ZZ}
\end{pmatrix}
 .
 \]

 To find $\Lambda_{\varepsilon}$, we observe that, by Theorem \ref{thm: duality}, $\pi^\varepsilon$ has the expression
\[
\begin{split}
d\pi^\varepsilon(u,x,y) &= \exp\left(\frac{1}{\varepsilon}\left(f^{\varepsilon}(u) + \langle g^\varepsilon(u), x\rangle  +h^{\varepsilon}(x,y) -\frac{1}{2}\|u-y\|^2\right)\right) \, d\mu(u)\,d\nu(x,y) \\
&=\frac{1}{(2\pi)^{(d_x+2d_y)/2}\sqrt{\det (\Sigma)}} \exp\left(\frac{1}{\varepsilon}\left(f^{\varepsilon}(u) + \langle g^\varepsilon(u), x\rangle  +h^{\varepsilon}(x,y) -\frac{1}{2}\|u-y\|^2\right)\right) \\
&\quad \times \exp \left ( -\frac{1}{2} \|u\|^2 - \frac{1}{2}(z-m)^\top \Sigma^{-1}(z-m) \right ) \, du dz,
\end{split}
\]
where $z = (x^\top, y^\top)^{\top}$. 
Inside the exponential function, the only cross term between $u$ and $y$ is 
$u^\top y/\varepsilon$. 
Comparing the densities on both sides, we have
\[
\frac{1}{\varepsilon}u^\top y = - \frac{1}{2}\left(u^\top \Theta_{UY}y  + y^\top \Theta_{YU}u\right) = -u^\top \Theta_{UY}y.
\]
This implies
\[
\Theta_{UY} = - \frac{1}{\varepsilon}I_{d_y}.
\]
In addition, from $\Gamma_\varepsilon \Theta = I_{d_x + 2d_y}$, the product of the second block row of $\Gamma_\varepsilon$ and the first block column of $\Theta$ is zero matrix:  $\Sigma_{XX} \Theta_{XU} + \Sigma_{XY} \Theta_{YU} = O$. 
Substituting $\Theta_{YU} = \Theta_{UY}^\top = -\frac{1}{\varepsilon} I_{d_y}$, we find 
\[
\Theta_{XU} = \frac{1}{\varepsilon} \Sigma_{XX}^{-1} \Sigma_{XY}.
\]

Now, we use the block matrix inversion formulas to relate $\Theta_{UY}$ to  $\Lambda_{\varepsilon}$. Partition $\Gamma_\varepsilon$ as 
\begin{equation}
\Gamma_\varepsilon = 
\begin{pmatrix}
I_{d_y} & B \\ B^\top & \Sigma
\end{pmatrix}
\quad \text{with} \quad B := \begin{pmatrix}
O & \Lambda_{\varepsilon}
\end{pmatrix}.
\label{eq: partition}
\end{equation}
As $\Sigma$ is positive definite, its Schur complement $I_{d_y} - B\Omega B^\top$ is positive definite (see, e.g., \cite[Corollary 3.1]{ouellette1981schur}). 
From the block matrix inversion formula  \cite[Exercise 5.16(b)]{abadir2005matrix}, we have 
\[
\Theta_{UZ} =- (I_{d_y} - B \Omega B^\top)^{-1} B \Omega \quad \text{with} \quad \Omega = 
\begin{pmatrix}
\Omega_{XX} & \Omega_{XY} \\
\Omega_{YX} & \Omega_{YY}
\end{pmatrix}
:=\Sigma^{-1}.
\]
Observe that $B\Omega = \begin{pmatrix}
\Lambda_{\varepsilon}\Omega_{YX} & \Lambda_{\varepsilon}\Omega_{YY}
\end{pmatrix}$ and $I_{d_y} - B \Omega B^\top = I_{d_y} - \Lambda_{\varepsilon} \Omega_{YY} \Lambda_{\varepsilon}^\top$. 
Extracting the block matrix corresponding to $Y$, we obtain
\[
\Theta_{UY} = - (I_{d_y} - \Lambda_{\varepsilon} \Omega_{YY} \Lambda_{\varepsilon}^\top)^{-1} \Lambda_{\varepsilon} \Omega_{YY}.
\]
Since  $\Theta_{UY} = - \frac{1}{\varepsilon}I_{d_y}$, we obtain the equation
\begin{equation}
\label{eq: Riccati}
\Lambda_{\varepsilon} \Omega_{YY} \Lambda_{\varepsilon}^\top + \varepsilon \Lambda_{\varepsilon} \Omega_{YY} - I_{d_y} = O.
\end{equation}
In particular, the equation implies that $\Lambda_{\varepsilon}\Omega_{YY}$ is symmetric. 
Recalling that $I_{d_y} - B \Sigma^{-1} B^\top= I_{d_y} - \Lambda_{\varepsilon} \Omega_{YY} \Lambda_{\varepsilon}^\top$ is positive definite, we conclude that $\Lambda_{\varepsilon} \Omega_{YY}$ is positive definite. 

Let $\Psi_\varepsilon := \Lambda_{\varepsilon} \Omega_{YY}$. If  $\Psi_\varepsilon$ and $\Omega_{YY}^{-1}$ commute, then $\Psi_\varepsilon\Omega_{YY}^{-1} = \Lambda_{\varepsilon}$ is symmetric positive definite, since the product of two commuting symmetric positive definite matrices is symmetric positive definite. We shall verify that $\Psi_\varepsilon$ and $\Omega_{YY}^{-1}$ indeed commute.  Plugging $\Lambda_{\varepsilon} = \Psi_\varepsilon \Omega_{YY}^{-1}$ into equation (\ref{eq: Riccati}), we see that $(\Psi_\varepsilon \Omega_{YY}^{-1}) \Omega_{YY} (\Psi_\varepsilon \Omega_{YY}^{-1})^\top + \varepsilon \Psi_\varepsilon = I_{d_y}$, or $\Psi_\varepsilon (\Psi_\varepsilon \Omega_{YY}^{-1})^\top + \varepsilon \Psi_\varepsilon = I_{d_y}$. Since  $\Psi_\varepsilon$ and $\Omega_{YY}^{-1}$ are symmetric, $\Psi_\varepsilon \Omega_{YY}^{-1} \Psi_\varepsilon + \varepsilon \Psi_\varepsilon = I_{d_y}$. Right-multiplying by $\Psi_\varepsilon^{-1}$ gives $\Psi_\varepsilon  \Omega_{YY}^{-1} + \varepsilon I_{d_y} = \Psi_\varepsilon^{-1}$. Left-multiplying by $\Psi_\varepsilon^{-1}$ gives $\Omega_{YY}^{-1} \Psi_\varepsilon + \varepsilon I_{d_y} = \Psi_\varepsilon^{-1}$. Comparing two, we see that $\Psi_\varepsilon  \Omega_{YY}^{-1} = \Omega_{YY}^{-1} \Psi_\varepsilon$, which yields that $\Lambda_{\varepsilon}$ is symmetric  positive definite. 

Finally, we observe that $\Omega_{YY}\Lambda_{\varepsilon} =\Omega_{YY} \Psi_\varepsilon\Omega_{YY}^{-1} = \Psi_\varepsilon = \Lambda_{\varepsilon}\Omega_{YY}$, so $\Lambda_{\varepsilon}$ and $\Omega_{YY}$ commute as well. 
Equation (\ref{eq: Riccati}) now reads as $\Lambda_{\varepsilon}^2 + \varepsilon \Lambda_{\varepsilon} = \Omega_{YY}^{-1}$, or
\[
\left(\Lambda_{\varepsilon} + \frac{\varepsilon}{2}I_{d_y}\right)^2 = \Omega_{YY}^{-1} + \frac{\varepsilon^2}{4}I_{d_y}.
\]
Solving this leads to 
\[
\Lambda_{\varepsilon} = \left(\Omega_{YY}^{-1} + \frac{\varepsilon^2}{4} I_{d_y}\right)^{1/2} - \frac{\varepsilon}{2} I_{d_y}.
\]
Using the block matrix inversion formula again, we have $\Omega_{YY}^{-1} = \Sigma_{YY} - \Sigma_{YX}\Sigma_{XX}^{-1}\Sigma_{XY}$. This completes the proof of Part (i).
\end{proof}

\subsection{Proof of Theorem \ref{thm: gaussian solution}~(ii)}
To find the dual potentials, we first need to find $\Theta$ explicitly. We have already obtained $\Theta_{UY} = - \frac{1}{\varepsilon}I_{d_y}$ and $\Theta_{XU} = \frac{1}{\varepsilon} \Sigma_{XX}^{-1} \Sigma_{XY}$. 

From the block matrix inversion formula applied to $\Theta = \Gamma_\varepsilon^{-1}$, we obtain  
\[
\begin{pmatrix}
\Theta_{UX} & \Theta_{UY}
\end{pmatrix} = - \Theta_{UU} \begin{pmatrix}
O & \Lambda_\varepsilon
\end{pmatrix} \begin{pmatrix}
\Omega_{XX} & \Omega_{XY} \\ \Omega_{YX} & \Omega_{YY}
\end{pmatrix}.
\]
Since $\Theta_{UY} = - \frac{1}{\varepsilon}I_{d_y}$,  we deduce that $\Theta_{UU} = \frac{1}{\varepsilon} (\Lambda_{\varepsilon} \Omega_{YY})^{-1}$. But according to equation (\ref{eq: Riccati}), $\Lambda_{\varepsilon} \Omega_{YY} (\Lambda_{\varepsilon} + \varepsilon I_{d_y}) = I_{d_y}$, which implies $(\Lambda_{\varepsilon} \Omega_{YY})^{-1} = \Lambda_{\varepsilon} + \varepsilon I_{d_y}$. Conclude that 
\[
\Theta_{UU} = \frac{1}{\varepsilon} (\Lambda_{\varepsilon} + \varepsilon I_{d_y}) = \frac{1}{\varepsilon} \Lambda_{\varepsilon} + I_{d_y}.
\]

It remains to find $\Theta_{ZZ} = (\Sigma - B^\top B)^{-1}$. Recall the partition (\ref{eq: partition}). From the Sherman-Morrison-Woodbury formula, we have
\[
(\Sigma - B^\top B)^{-1} = \Sigma^{-1} + \Sigma^{-1}B^\top (I_{d_y} - B\Sigma^{-1}B^\top)^{-1}B\Sigma^{-1}
\]
Recalling that $\Lambda_\varepsilon$ is symmetric, we observe that
\[
\Sigma^{-1}B^\top = \begin{pmatrix}
\Omega_{XY} \Lambda_{\varepsilon} \\
\Omega_{YY} \Lambda_{\varepsilon}
\end{pmatrix} \quad \text{and} \quad 
I_{d_y}- B\Sigma^{-1}B^\top = I_{d_y} - \Lambda_{\varepsilon}\Omega_{YY}\Lambda_{\varepsilon}^\top.
\]
Since, by equation (\ref{eq: Riccati}), $I_{d_y} - \Lambda_{\varepsilon}\Omega_{YY}\Lambda_{\varepsilon} = \varepsilon \Lambda_{\varepsilon}\Omega_{YY}$, we see that
\[
(I - B\Sigma^{-1}B^\top )^{-1} = \frac{1}{\varepsilon}\Omega_{YY}^{-1}\Lambda_{\varepsilon}^{-1}.
\]
Combining these yields
\[
\Theta_{ZZ} = \Omega + \frac{1}{\varepsilon} \begin{pmatrix} \Omega_{XY} \Lambda_{\varepsilon} \\ \Omega_{YY} \Lambda_{\varepsilon} \end{pmatrix} (\Omega_{YY}^{-1} \Lambda_{\varepsilon}^{-1}) \begin{pmatrix} \Lambda_{\varepsilon} \Omega_{YX} & \Lambda_{\varepsilon} \Omega_{YY} \end{pmatrix}.
\]
The expression can be further simplified.
From $\Sigma \Omega = I_{d_x+d_y}$ we have $\Sigma_{XX} \Omega_{XY} + \Sigma_{XY} \Omega_{YY} = O$, or $\Omega_{XY} = G\Omega_{YY}$ with $G = -\Sigma_{XX}^{-1} \Sigma_{XY}$.
Plugging this, we have
\[
\begin{split}
\Theta_{ZZ} &= \Omega + \frac{1}{\varepsilon} \left( \begin{pmatrix} G \\ I_{d_y} \end{pmatrix} \Omega_{YY} \Lambda_{\varepsilon} \right) (\Omega_{YY}^{-1} \Lambda_{\varepsilon}^{-1}) \left( \Lambda_{\varepsilon} \Omega_{YY} \begin{pmatrix} G^\top & I_{d_y} \end{pmatrix} \right) \\
&= \Omega + \frac{1}{\varepsilon} \begin{pmatrix} G \\ I_{d_y} \end{pmatrix} (\Lambda_{\varepsilon} \Omega_{YY}) \begin{pmatrix} G^\top & I_{d_y} \end{pmatrix} \\
&= \Omega + \frac{1}{\varepsilon} \begin{pmatrix} G \Psi_\varepsilon G^\top & G \Psi_\varepsilon \\ \Psi_\varepsilon G^\top & \Psi_\varepsilon \end{pmatrix}
\end{split}
\]
where $\Psi_\varepsilon = \Lambda_{\varepsilon}\Omega_{YY}$. 

In conclusion, denoting by 
$\Theta_0 := \begin{pmatrix}
I_{d_y} & O \\
O & \Omega
\end{pmatrix}$
the precision matrix for the reference measure $\mu \otimes \nu$, we have
\[
\Delta \Theta := \Theta - \Theta_0 = \frac{1}{\varepsilon} \begin{pmatrix}
\Lambda_{\varepsilon} & -G^\top & -I_{d_y} \\
-G & G \Psi_\varepsilon G^\top & G \Psi_\varepsilon \\
-I_{d_y} & \Psi_\varepsilon G^\top & \Psi_\varepsilon
\end{pmatrix}.
\]
Since $\pi^\varepsilon$ is multivariate Gaussian, for $w = (u^\top,x^\top,y^\top)^\top$,
\[
\begin{split}
\log\frac{d\pi}{d(\mu\otimes \nu)}(w) &= -\frac{1}{2}(w-m)^\top \Theta (w -m) + \frac{1}{2}(w-m)^\top \Theta_0(w-m) + \frac{1}{2}\log\frac{\det (\Theta)}{\det (\Theta_0)} \\
\end{split}
\]

We shall find the constant $\frac{1}{2}\log\frac{\det (\Theta)}{\det (\Theta_0)}$. 
Observe that $\det(\Theta) = \det(\Gamma_\varepsilon)^{-1}$ and $\det(\Theta_0) = \det(\Sigma)^{-1}$. From the partition (\ref{eq: partition}),  Schur's formula yields
\[
\begin{split}
    \det(\Gamma_\varepsilon) &= \det(\Sigma) \det(I_{d_y} - \Lambda_{\varepsilon} \Omega_{YY} \Lambda_{\varepsilon}) \\
    &= \det(\Sigma) \det(\varepsilon \Lambda_{\varepsilon} \Omega_{YY})
\end{split} 
\]
where the last step follows from equation (\ref{eq: Riccati}). Therefore,
\[
\frac{\det (\Theta)}{\det (\Theta_0)} = \frac{\det(\Sigma)^{-1} \det(\varepsilon \Lambda_{\varepsilon} \Omega_{YY})^{-1}}{\det(\Sigma)^{-1}} = \det(\varepsilon \Lambda_{\varepsilon} \Omega_{YY})^{-1}.
\]

It remains to find explicit formulas for $(f^\varepsilon, g^\varepsilon, h^\varepsilon)$ so that
\begin{equation}
\begin{split}
&\frac{1}{\varepsilon}\left(f^\varepsilon(u) + \langle g^\varepsilon(u), x\rangle + h^\varepsilon(x,y) - \frac{1}{2}\|u-y\|^2\right) \\
&\quad = -\frac{1}{2}(w-m)^\top \Delta \Theta (w-m) - \frac{1}{2}\log \det(\varepsilon \Lambda_{\varepsilon} \Omega_{YY}).
\end{split}
\label{eq: density gaussian}
\end{equation}
By direct calculation, 
\[
\begin{split}
\varepsilon(w-m)^\top \Delta \Theta (w-m) &= u^\top \Lambda_{\varepsilon} u + x^\top G\Psi_\varepsilon G^\top x + (y-m_Y)^\top \Psi_\varepsilon(y-m_Y) \\
&\qquad -2u^\top G^\top x -2u^\top (y-m_Y) + 2x^\top G\Psi_\varepsilon(y-m_Y).
\end{split}
\]
It is straightforward to verify that choosing $(f^\varepsilon,g^\varepsilon,h^\varepsilon)$ as in (\ref{eq: gaussian dual potentials}) satisfies (\ref{eq: density gaussian}).
\qed

\subsection{Proof of Proposition \ref{prop: gaussian limit}}

The weak convergence follows directly from the closed-form expression of $\pi^\varepsilon$. 
For $(U,\tilde{X},\tilde{Y}) \sim \pi^o$, the conditional distribution of $\tilde{Y}$ given $(U,\tilde{X})$ is a multivariate Gaussian distribution with mean 
\[
m_Y + 
\begin{pmatrix}
\Lambda_o & \Sigma_{YX} 
\end{pmatrix}
\begin{pmatrix}
I_{d_y} & O \\
O & \Sigma_{XX}^{-1}
\end{pmatrix}
\begin{pmatrix}
U \\
\tilde{X}
\end{pmatrix}
=m_Y + \Lambda_o U + \Sigma_{YX}\Sigma_{XX}^{-1}\tilde{X},
\]
and covariance matrix
\[
\Sigma_{YY}-\begin{pmatrix}
\Lambda_o & \Sigma_{YX} 
\end{pmatrix}
\begin{pmatrix}
I_{d_y} & O \\
O & \Sigma_{XX}^{-1}
\end{pmatrix}
\begin{pmatrix}
\Lambda_o \\ \Sigma_{YX} 
\end{pmatrix}
=\Sigma_{YY}-\Lambda_o^2-\Sigma_{YX}\Sigma_{XX}^{-1}\Sigma_{XY} = O,
\]
that is, $\tilde{Y} = m_Y + \Lambda_o U + \Sigma_{YX}\Sigma_{XX}^{-1}\tilde{X}$ a.s. In view of Example 2.1 in \cite{carlier2016vector}, $\pi^o$ is optimal for (\ref{eq: vqr indep}).
\qed

\subsection{Proof of Proposition \ref{prop: gaussian approx}}

    Assume without loss of generality that $m_Y = 0$. We write $w=(u,x,y) \in \R^{N}$ with $N:=d_x+2d_y$ instead of $w=(u^\top,x^\top,y^\top)^\top$ for notational convenience.     Observe that $\pi^o$ is supported on the vector subspace
    \[
    \mathcal{S} := \left\{(u,x,y) \in \R^{N} :y = \Lambda_o u  -G^\top x \right\}.
    \]
    Let $\calS^\bot$ denote the orthocomplement of $\calS$ and $\mathsf{P}_{\calS^\bot}$ denote the projection matrix onto $\calS^\bot$ with $\mathsf{P}_{\calS}:=I_{N}-\mathsf{P}_{\calS^\bot}$. 
    We observe that for any $W^\varepsilon \sim \pi^\varepsilon$ and $W^o \sim \pi^o$, 
    \[
    \| W^\varepsilon-W^o \|^2 = \| \mathsf{P}_{\calS^\bot}W^\varepsilon\|^2 + \| \mathsf{P}_{\calS}W^\varepsilon-W^o\|^2.
    \]
    This implies 
    \[
    \mathsf{W}_2^2(\pi^\varepsilon,\pi^o) = \E[\| \mathsf{P}_{\calS^\bot}W^\varepsilon\|^2] + \mathsf{W}_2^2(\pi^\varepsilon \circ \mathsf{P}_{\calS}^{-1},\pi^o). 
    \]
 The first term on the right-hand side depends only on the marginal $\pi^\varepsilon$. We split the rest of the proof into two steps.

 \underline{Step 1}. We first establish that 
 \[
 \E[\| \mathsf{P}_{\calS^\bot}W^\varepsilon\|^2] = \varepsilon \tr (L^{-1}\Lambda_o) + O(\varepsilon^{2}). 
 \]
Let $U \sim \calN(0,I_{d_y})$ and $X \sim \calN(0,\Sigma_{XX})$ be independent and generate $Y^\varepsilon$ as
\[
Y^\varepsilon = \Lambda_\varepsilon U -G^\top X + \eta_\varepsilon, \quad \eta_\varepsilon \sim \calN(0,\varepsilon \Lambda_\varepsilon), \ \eta_\varepsilon \independent (U,X).
\]
It is not difficult to see that $W^\varepsilon := (U,X,Y^\varepsilon) \sim \pi^\varepsilon$.    We have
\[
W^\varepsilon = \underbrace{(U,X,\Lambda_o U - G^\top X)}_{\in \calS} + (0,0,(\Lambda_\varepsilon-\Lambda_o)U+\eta_\varepsilon),
\]
so that $\mathsf{P}_{\calS^\bot}W^\varepsilon = \mathsf{P}_{\calS^\bot}(0,0,(\Lambda_\varepsilon-\Lambda_o)U+\eta_\varepsilon)$. 
We shall verify the following lemma.

    \begin{lemma}
        \label{lem: distance formula}
        For any $r \in \R^{d_y}$, the distance from $(0,0,r) \in \R^{N}$ to $\calS$ is $\sqrt{r^\top L^{-1}r}$.
    \end{lemma}
    \begin{proof}[Proof of Lemma \ref{lem: distance formula}]
        Observe that $\mathcal{S} = \{w \in \R^{N}: Hw = 0\}$ with $H:= \begin{pmatrix}
            \Lambda_o & -G^\top & -I_{d_y}
        \end{pmatrix}$, so that $\mathsf{P}_{\calS^\bot}=H^\top (H H^\top)^{-1} H$ with $HH^\top =  L$. Since $H(0,0,r) = -r$, the desired claim follows.
        \end{proof}
Now, we have
\[
\begin{split}
\| \mathsf{P}_{\calS^\bot}(0,0,(\Lambda_\varepsilon-\Lambda_o)U+\eta_\varepsilon)\|^2 &= \big((\Lambda_\varepsilon-\Lambda_o)U+\eta_\varepsilon\big)^\top L^{-1}\big((\Lambda_\varepsilon-\Lambda_o)U+\eta_\varepsilon\big) \\
&=\tr \Big( L^{-1} \big((\Lambda_\varepsilon-\Lambda_o)U+\eta_\varepsilon\big)\big((\Lambda_\varepsilon-\Lambda_o)U+\eta_\varepsilon\big)^\top \Big).
\end{split}
\]
Taking expectation, we have
\[
\begin{split}
\E[\| \mathsf{P}_{\calS^\bot}W^\varepsilon\|^2] &= \tr \Big ( L^{-1} \big ( (\Lambda_\varepsilon - \Lambda_o)^2 + \varepsilon \Lambda_\varepsilon \big) \Big) \\
&=\varepsilon \tr (L^{-1}\Lambda_o) + \tr \big ( L^{-1}  (\Lambda_\varepsilon - \Lambda_o)^2 \big ) + \varepsilon \tr \big(L^{-1}(\Lambda_\varepsilon - \Lambda_o)\big). 
\end{split}
\]
Using spectral expansions, one can see that $\Lambda_\varepsilon  - \Lambda_o = O(\varepsilon)$, so that the last two terms on the right-hand side are $O(\varepsilon^2)$. This finishes the proof of Step 1. 

\underline{Step 2}. In this step, we shall show that 
\[
\mathsf{W}_2^2(\pi^\varepsilon \circ \mathsf{P}_{\calS}^{-1},\pi^o) = O(\varepsilon^2),
\]
which, combined with Step 1, leads to the conclusion of the proposition. 

        Let $\mathsf{Q} \in \R^{N \times(d_x + d_y)}$ be a matrix whose columns consist of orthonormal vectors spanning $\mathcal{S}$, so that $\mathsf{P}_{\calS} = \mathsf{Q} \mathsf{Q}^\top$. Observe that 
        \[
    \mathsf{W}_2^2(\pi^\varepsilon \circ \mathsf{P}_{\calS}^{-1},\pi^o) = \inf_{W^\varepsilon \sim \pi^\varepsilon,W^o \sim \pi^o} \E[\| \mathsf{Q}^\top W^\varepsilon - \mathsf{Q}^\top W^o\|^2]. 
        \]
        We have $\mathsf{Q}^\top W^\varepsilon \sim \calN(0,\mathsf{Q}^\top \Gamma_\varepsilon \mathsf{Q})$ and $\mathsf{Q}^\top W^o \sim \calN(0,\mathsf{Q}^\top \Gamma_o \mathsf{Q})$. Now, \cite[Exercise 1.8.3(b)]{chewi2025statistical} yields
        \[
      \inf_{W^\varepsilon \sim \pi^\varepsilon,W^o \sim \pi^o} \E[\| \mathsf{Q}^\top W^\varepsilon - \mathsf{Q}^\top W^o\|^2] \le \| (\mathsf{Q}^\top\Gamma_\varepsilon\mathsf{Q})^{1/2} - (\mathsf{Q}^\top\Gamma_o \mathsf{Q})^{1/2}\|_{\text{F}}^2,
        \]
        where $\| \cdot \|_{\mathrm{F}}$ denotes the Frobenius norm. 
       Observe that $\| \mathsf{Q}^\top(\Gamma_\varepsilon - \Gamma_o) \mathsf{Q}\|_{\text{F}} \leq \|\Gamma_\varepsilon - \Gamma_o\|_{\text{F}}$.
        Recalling the explicit formulas for $\Gamma_\varepsilon$ and $\Gamma_o$, we have $\|\Gamma_\varepsilon - \Gamma_o\|_{\text{F}}^2 = 2\|\Lambda_\varepsilon -\Lambda_o\|_{\text{F}}^2 \le d_y\varepsilon^2 / 2$. Since $\lambda_{\min} ((\mathsf{Q}^\top \Gamma_o \mathsf{Q})^{1/2}) = \sqrt{\lambda_{\min} (\mathsf{Q}^\top \Gamma_o \mathsf{Q})} > 0$ (with $\lambda_{\min}(\cdot)$ denoting the minimum eigenvalue), \cite[Problem X.5.5]{bhatia1997matrix} yields
        \[
        \| (\mathsf{Q}^\top\Gamma_\varepsilon\mathsf{Q})^{1/2} - (\mathsf{Q}^\top\Gamma_o \mathsf{Q})^{1/2}\|_{\text{F}} \le \frac{1}{\sqrt{\lambda_{\min} (\mathsf{Q}^\top \Gamma_o \mathsf{Q})}} \| \mathsf{Q}^\top(\Gamma_\varepsilon - \Gamma_o) \mathsf{Q}\|_{\text{F}}.
        \]
      Conclude that 
        \[
\mathsf{W}_2^2(\pi^\varepsilon \circ \mathsf{P}_{\calS}^{-1},\pi^o)  \le \frac{1}{2\lambda_{\min} (\mathsf{Q}^\top \Gamma_o \mathsf{Q})} d_y \varepsilon^2 = O(\varepsilon^2),
        \]
completing the proof.
\qed

\appendix
\section{Partial converse of Lemma \ref{lem: coercive}}
\label{sec: converse}

The following lemma provides a partial converse of Lemma \ref{lem: coercive}. 

\begin{lemma}
Let $(M,d)$ be a Polish metric space.
Suppose that $Q \in \calP_1(M)$, but
\[
\int e^{\alpha d(x,x_0)} \, dQ(x) = \infty
\]
for every $\alpha > 0$ for some $x_0 \in M$. Then there exists a sequence $P_n \in \calP_1(M)$ such that $\limsup_{n \to \infty} \KL{P_n}{Q} < \infty$ but $\lim_{n \to \infty} \mathsf{W}_1(P_n,Q) = \infty$. In particular, $\KL{\cdot}{Q}$ is not coercive in $\mathsf{W}_1$.
\end{lemma}

\begin{proof}
Set $\bar{F}(r):= Q(\{ d(\cdot,x_0) \ge r \})$. We first verify that 
\[
\liminf_{r \to \infty}r^{-1}\log \big(1/\bar{F}(r)\big) = 0.
\]
Suppose on the contrary that the left-hand side is strictly positive (including $\infty$). Then, there exist $\beta > 0, r_0 > 0$ such that 
\[
\bar{F}(r) \le e^{-\beta r}, \ r \ge r_0,
\]
which, however, implies that $d(\cdot,x_0)$ is sub-exponential under $Q$, contradicting the assumption. 

Now, choose $r_n \to \infty$ such that 
\[
\lim_{n \to \infty} r_n^{-1}q_n= 0 \quad \text{with} \quad q_n := \log \big(1/\bar{F}(r_n)\big). 
\]
Assuming that $n$ is large enough, we define 
\[
P_n := (1-q_n^{-1})Q + q_n^{-1}Q_n \quad \text{with} \quad dQ_n := e^{q_n}\mathbbm{1}_{\{ d(\cdot,x_0) \ge r_n \}} \, dQ. 
\]
By convexity of $\KL{\cdot}{Q}$, 
\[
\KL{P_n}{Q} \le q_n^{-1} \KL{Q_n}{Q} = 1.
\]
On the other hand, by the Kantorovich-Rubinstein duality,
\[
\begin{split}
\mathsf{W}_1(P_n,Q) &\ge \int d(\cdot,x_0) \, d(P_n-Q) = q_n^{-1} \int d(\cdot,x_0) \, d(Q_n-Q) \\
&\ge q_n^{-1}(r_n-a),
\end{split}
\]
where $a := \int d(\cdot,x_0) \, dQ < \infty$. Since $\lim_{n \to \infty} q_n^{-1}r_n = \infty$ by construction, we obtain the desired claim. 
\end{proof}
\bibliographystyle{alpha}
\bibliography{reference}

@book{koenker2005quantile,
  title={Quantile Regression},
  author={Koenker, R.},
  year={2005},
  publisher={Cambridge University Press}
}

@article{sinkhorn1967diagonal,
  title={Diagonal equivalence to matrices with prescribed row and column sums},
  author={Sinkhorn, R.},
  journal={The American Mathematical Monthly},
  volume={74},
  number={4},
  pages={402--405},
  year={1967},
  publisher={JSTOR}
}

@article{franklin1989scaling,
  title={On the scaling of multidimensional matrices},
  author={Franklin, J. and Lorenz, J.},
  journal={Linear Algebra and its Applications},
  volume={114},
  pages={717--735},
  year={1989},
  publisher={Elsevier}
}

@article{gonzalez2022weak,
  title={Weak limits of entropy regularized optimal transport; potentials, plans and divergences},
  author={Gonzalez-Sanz, A. and Loubes, J.-M. and Niles-Weed, J.},
  journal={arXiv preprint arXiv:2207.07427},
  year={2022}
}

@article{genevay2019sample,
  title={Sample complexity of {S}inkhorn divergences},
  author={Genevay, A. and Chizat, L. and Bach, F. and Cuturi, M. and Peyr{\'e}, G.},
  journal={Proceedings of the 22nd International Conference on Artificial Intelligence and Statistics},
  pages={1574--1583},
  year={2019},
  volune={89},
  organization={PMLR}
}

@article{goldfeld2024limit,
  title={Limit theorems for entropic optimal transport maps and Sinkhorn divergence},
  author={Goldfeld, Z. and Kato, K. and Rioux, G. and Sadhu, R.},
  journal={Electronic Journal of Statistics},
  volume={18},
  number={1},
  pages={980--1041},
  year={2024},
  publisher={The Institute of Mathematical Statistics and the Bernoulli Society}
}

@incollection{follmer1988random,
  title={Random fields and diffusion processes},
  author={F{\"o}llmer, H.},
  booktitle={{\'E}cole d'{\'E}t{\'e} de Probabilit{\'e}s de Saint-Flour XV--XVII, 1985--87},
  pages={101--203},
  year={1988},
  publisher={Springer}
}

@article{nutz2022entropic,
author = {Nutz, M. and Wiesel, J.},
  title={Entropic optimal transport: Convergence of potentials},
  journal={Probability Theory and Related Fields},
  volume={184},
  number={1},
  pages={401--424},
  year={2022},
  publisher={Springer}
}

@book{deuschel2001large,
  title={Large Deviations},
  author={Deuschel, J.-D. and Stroock, D. W.},
  year={2001},
  publisher={American Mathematical Society}
}

@article{wang2010sanov,
  title={Sanov’s theorem in the {W}asserstein distance: a necessary and sufficient condition},
  author={Wang, R. and Wang, X. and Wu, L.},
  journal={Statistics \& Probability Letters},
  volume={80},
  number={5-6},
  pages={505--512},
  year={2010},
  publisher={Elsevier}
}

@article{hallin2022measure,
  title={Measure transportation and statistical decision theory},
  author={Hallin, M.},
  journal={Annual Review of Statistics and Its Application},
  volume={9},
  number={1},
  pages={401--424},
  year={2022},
  publisher={Annual Reviews}
}

@article{ghosal2022multivariate,
  title={Multivariate ranks and quantiles using optimal transport: Consistency, rates and nonparametric testing},
  author={Ghosal, P. and Sen, B.},
  journal={The Annals of Statistics},
  volume={50},
  number={2},
  pages={1012--1037},
  year={2022},
  publisher={Institute of Mathematical Statistics}
}

@article{chernozhukov2017monge,
  title={Monge--{K}antorovich depth, quantiles, ranks and signs},
  author={Chernozhukov, V. and Galichon, A. and Hallin, M. and Henry, M.},
  journal={The Annals of Statistics},
  volume={45},
  number={1},
  pages={223--256},
  year={2017},
}

@article{koenker1978regression,
  title={Regression quantiles},
  author={Koenker, R. and Bassett Jr, G.},
  journal={Econometrica},
  volume={46},
  number={1},
  pages={33--50},
  year={1978},
  publisher={JSTOR}
}

@article{janati2020entropic,
  title={Entropic optimal transport between unbalanced {G}aussian measures has a closed form},
  author={Janati, H. and Muzellec, B. and Peyr{\'e}, G. and Cuturi, M.},
  journal={Advances in Neural information Processing Systems},
  volume={33},
  pages={10468--10479},
  year={2020}
}

@article{malamut2025convergence,
  title={Convergence rates of the regularized optimal transport: Disentangling suboptimality and entropy},
  author={Malamut, H. and Sylvestre, M.},
  journal={SIAM Journal on Mathematical Analysis},
  volume={57},
  number={3},
  pages={2533--2558},
  year={2025},
  publisher={SIAM}
}

@article{mena2019statistical,
  title={Statistical bounds for entropic optimal transport: sample complexity and the central limit theorem},
  author={Mena, G. and Niles-Weed, J.},
  journal={Advances in Neural Information Processing Systems},
  volume={32},
  year={2019}
}

@article{cuturi2013sinkhorn,
  title={Sinkhorn distances: Lightspeed computation of optimal transport},
  author={Cuturi, M.},
  journal={Advances in Neural Information Processing Systems},
  volume={26},
  year={2013}
}

@article{peyre2019computational,
  title={Computational optimal transport: With applications to data science},
  author={Peyr{\'e}, G. and Cuturi, M.},
  journal={Foundations and Trends{\textregistered} in Machine Learning},
  volume={11},
  number={5-6},
  pages={355--607},
  year={2019},
  publisher={Now Publishers, Inc.}
}

@book{santambrogio2015optimal,
  title={Optimal Transport for Applied Mathematicians},
  author={Santambrogio, F.},
  year={2015},
  publisher={Springer}
}

@article{brenier1991polar,
  title={Polar factorization and monotone rearrangement of vector-valued functions},
  author={Brenier, Y.},
  journal={Communications on Pure and Applied Mathematics},
  volume={44},
  number={4},
  pages={375--417},
  year={1991},
  publisher={Wiley Online Library}
}

@article{follmer1997entropy,
  title={Entropy minimization and {S}chr{\"o}dinger processes in infinite dimensions},
  author={F{\"o}llmer, H. and Gantert, N.},
  journal={The Annals of Probability},
  volume={25},
  number={2},
  pages={901--926},
  year={1997},
  publisher={Institute of Mathematical Statistics}
}

@article{nutz2024martingale,
  title={On the Martingale {S}chr\"{o}dinger Bridge between Two Distributions},
  author={Nutz, M. and Wiesel, J.},
  journal={arXiv preprint arXiv:2401.05209},
  year={2024}
}

@book{krantz2002primer,
  title={A Primer of Real Analytic Functions},
  author={Krantz, S. G. and Parks, H. R.},
  edition={Second},
  year={2002},
  publisher={Springer Science \& Business Media}
}

@article{carlier2016vector,
  title={Vector quantile regression: an optimal transport approach},
  author={Carlier, G. and Chernozhukov, V. and Galichon, A.},
  journal={The Annals of Statistics},
  volume={44},
  number={3},
  pages={1165--1192},
  year={2016}
}

@article{carlier2022vector,
  title={Vector quantile regression and optimal transport, from theory to numerics},
  author={Carlier, G. and Chernozhukov, V. and De Bie, G. and Galichon, A.},
  journal={Empirical Economics},
  volume={62},
  number={1},
  pages={35--62},
  year={2022},
  publisher={Springer}
}

@book{chewi2025statistical,
  title={Statistical Optimal Transport},
  author={Chewi, S. and Niles-Weed, J. and Rigollet, P.},
  series={Lecture Notes in Mathematics},
  year = {2025},
  publisher={Springer}
}

@article{carlier2023convergence,
  title={Convergence rate of general entropic optimal transport costs},
  author={Carlier, G. and Pegon, P. and Tamanini, L.},
  journal={Calculus of Variations and Partial Differential Equations},
  volume={62},
  number={4},
  pages={116},
  year={2023},
  publisher={Springer}
}

@article{carlier2017vector,
  title={Vector quantile regression beyond the specified case},
  author={Carlier, G. and Chernozhukov, V. and Galichon, A.},
  journal={Journal of Multivariate Analysis},
  volume={161},
  pages={96--102},
  year={2017},
  publisher={Elsevier}
}

@article{eckstein2024convergence,
  title={Convergence rates for regularized optimal transport via quantization},
  author={Eckstein, S. and Nutz, M.},
  journal={Mathematics of Operations Research},
  volume={49},
  number={2},
  pages={1223--1240},
  year={2024},
  publisher={INFORMS}
}

@article{ghosal2025convergence,
  title={On the convergence rate of {S}inkhorn’s algorithm},
  author={Ghosal, P. and Nutz, M.},
  journal={Mathematics of Operations Research},
  year={2025},
  publisher={INFORMS}
}

@book{panaretos2020invitation,
  title={An Invitation to Statistics in Wasserstein Space},
  author={Panaretos, V. M. and Zemel, Y.},
  series={SpringerBriefs in Probability and Mathematical Statistics},
  year={2020},
  publisher={Springer Nature}
}

@article{beiglbock2013model,
  title={Model-independent bounds for option prices—a mass transport approach},
  author={Beiglb{\"o}ck, M. and Henry-Labordere, P. and Penkner, F.},
  journal={Finance and Stochastics},
  volume={17},
  number={3},
  pages={477--501},
  year={2013},
  publisher={Springer}
}

@article{galichon2014stochastic,
  title={A stochastic control approach to No-Arbitrage bounds given marginals, with an application to Lookback options},
  author={Galichon, A. and Henri-Labord{\`e}re, P. and Touzi, N.},
  journal={The Annals of Applied Probability},
  volume={24},
  number={1},
  pages={312--336},
  year={2014}
}

@article{beiglbock2025fundamental,
  title={The Fundamental Theorem of Weak Optimal Transport},
  author={Beiglb{\"o}ck, M. and Pammer, G. and Riess, L. and Schrott, S.},
  journal={arXiv preprint arXiv:2501.16316},
  year={2025}
}

@article{leonard2012schrodinger,
  title={From the {S}chr{\"o}dinger problem to the {M}onge--{K}antorovich problem},
  author={L{\'e}onard, C.},
  journal={Journal of Functional Analysis},
  volume={262},
  number={4},
  pages={1879--1920},
  year={2012},
  publisher={Elsevier}
}

@article{mikami2004monge,
  title={Monge’s problem with a quadratic cost by the zero-noise limit of $h$-path processes},
  author={Mikami, T.},
  journal={Probability theory and related fields},
  volume={129},
  number={2},
  pages={245--260},
  year={2004},
  publisher={Springer}
}

@article{leonard2013survey,
  title={A survey of the {S}chr{\"o}dinger problem and some of its connections with optimal transport},
  author={L{\'e}onard, C.},
  journal={Discrete and Continuous Dynamical Systems},
  volume={34},
  number={4},
  pages={1533--1574},
  year={2013},
  publisher={Discrete and Continuous Dynamical Systems}
}

@article{nutz2021introduction,
  title={Introduction to Entropic Optimal Transport},
  author={Nutz, M.},
  journal={Lecture Notes, Columbia University},
  year={2021}
}

@book{villani2009optimal,
  title={Optimal Transport: Old and New},
  author={Villani, C.},
  year={2009},
  publisher={Springer}
}

@article{csiszar1975divergence,
  title={I-divergence geometry of probability distributions and minimization problems},
  author={Csisz{\'a}r, I.},
  journal={The Annals of Probability},
  volume={3},
  number={1},
  pages={146--158},
  year={1975}
}

@book{Dudley_2002, 
  title={Real Analysis and Probability}, 
  publisher={Cambridge University Press}, 
  author={Dudley, R. M.}, 
  edition = {second},
  year={2002}
}

@book{nesterov2018lectures,
  title={Lectures on Convex Optimization},
  author={Nesterov, Y.},
  year={2018},
  publisher={Springer}
}

@book{brezis2011functional,
  title={Functional Analysis, Sobolev Spaces and Partial Differential Equations},
  author={Brezis, H.},
  year={2011},
  publisher={Springer}
}

@article{wainwright2008graphical,
  title={Graphical models, exponential families, and variational inference},
  author={Wainwright, M. J. and Jordan, M. I.},
  journal={Foundations and Trends{\textregistered} in Machine Learning},
  volume={1},
  number={1--2},
  pages={1--305},
  year={2008},
  publisher={Now Publishers, Inc.}
}

@book{barndorff2014information,
  title={Information and Exponential Families: in Statistical Theory},
  author={Barndorff-Nielsen, O.},
  year={2014},
  publisher={John Wiley \& Sons}
}

@article{del2023improved,
  title={An improved central limit theorem and fast convergence rates for entropic transportation costs},
  author={del Barrio, E. and Gozalez-Sanz, A. and Loubes, J.-M. and Niles-Weed, J.},
  journal={SIAM Journal on Mathematics of Data Science},
  volume={5},
  number={3},
  pages={639--669},
  year={2023},
  publisher={SIAM}
}

@article{carlier2025weak,
  title={Weak optimal transport with moment constraints: constraint qualification, dual attainment and entropic regularization},
  author={Carlier, G. and Malamut, H. and Sylvestre, M.},
  journal={arXiv preprint arXiv:2511.16211},
  year={2025}
}

@article{mallasto2022entropy,
  title={Entropy-regularized 2-{W}asserstein distance between {G}aussian measures},
  author={Mallasto, A. and Gerolin, A. and Minh, H. Q.},
  journal={Information Geometry},
  volume={5},
  number={1},
  pages={289--323},
  year={2022},
  publisher={Springer}
}

@book{abadir2005matrix,
  title={Matrix Algebra},
  author={Abadir, K. M. and Magnus, J. R.},
  year={2005},
  publisher={Cambridge University Press}
}

@article{ouellette1981schur,
  title={Schur complements and statistics},
  author={Ouellette, D. V.},
  journal={Linear Algebra and its Applications},
  volume={36},
  pages={187--295},
  year={1981},
  publisher={Elsevier}
}

@book{polyanskiy2025information,
  title={Information Theory: From Coding to Learning},
  author={Polyanskiy, Y. and Wu, Y.},
  year={2025},
  publisher={Cambridge University Press}
}

@article{bolley2005weighted,
  title={Weighted {C}sisz{\'a}r-{K}ullback-{P}insker inequalities and applications to transportation inequalities},
  author={Bolley, F. and Villani, C.},
  journal={Annales de la Facult{\'e} des Sciences de Toulouse: Math{\'e}matiques},
  volume={14},
  number={3},
  pages={331--352},
  year={2005}
}

@book{bhatia1997matrix,
  title={Matrix Analysis},
  author={Bhatia, R.},
  year={1997},
  publisher={Springer}
}

\end{document}